\documentclass[12pt]{article}    
    
\usepackage{latexsym,amsfonts,amsmath,amssymb,wasysym,stmaryrd,enumerate,epsfig}    
\usepackage{mathrsfs}    
\usepackage{theorem}    
\usepackage[utf8]{inputenc}    
\usepackage{graphicx}    
\usepackage{pstricks}    
\usepackage{amsmath}    
\usepackage{tikz}    
\usetikzlibrary{matrix,arrows,shapes}    
    
\numberwithin{equation}{section}

\newtheorem{definition}{Definition}[section]

\newtheorem{proposition}[definition]{Proposition}    
\newtheorem{theorem}[definition]{Theorem}    
\newtheorem{corollary}[definition]{Corollary}    
\newtheorem{lemma}[definition]{Lemma}

\newcommand{\be}{\begin{equation}}    
\newcommand{\ee}{\end{equation}}    
\newcommand{\beu}{\begin{equation*}}    
\newcommand{\eeu}{\end{equation*}}    
\newcommand{\bea}{\begin{eqnarray}}    
\newcommand{\eea}{\end{eqnarray}}    
\newcommand{\beaa}{\begin{eqnarray*}}    
\newcommand{\eeaa}{\end{eqnarray*}}    
\newcommand{\bmx}{\begin{pmatrix}}    
\newcommand{\emx}{\end{pmatrix}}

\newcommand{\g}{{\mathfrak g}}

\newcommand{\finproof}{{\hfill \rule{5pt}{5pt}}}

\newcommand{\nn}{\nonumber}

\newcommand{\rank}{{\rm rank}}

\newcommand{\ket}[1]{{\,\left|#1\right>}\,}

\newcommand{\bl}{{\bullet}}    
\newcommand{\wh}{{\circ}}    
\newcommand{\disjoint}{\sqcup}    
\newcommand{\Zh}{\mathbb Z/h \mathbb Z}    
\newcommand{\bw}{{\stackrel{\bl}{\wh}}}    
\newcommand{\wb}{{\stackrel{\wh}{\bl}}}    
\newcommand{\nbr}[2]{\langle #1,#2 \rangle}    
\newcommand{\uqsl}[1]{{U_q(\widehat{\mathfrak{sl}}_2{}^{(#1)})}}    
\newcommand{\uqslt}{{  U_q(\widehat{\mathfrak{sl}}_2})}    
\newcommand{\uqgh}{{U_q(\widehat\g)}}    
\newcommand{\uqg}{{U_q(\g)}}    
\newcommand{\Yg}{{Y(\g)}}    
    
\newcommand{\YY}[2]{Y_{#1,aq^{#2}}}    
\newcommand{\MM}[2]{Y_{#1,aq^{#2}}^{-1}}    
\newcommand{\cq}{\chi_q(V_{i,a})}    
\newcommand{\cqo}{\chi_q(V_{i_1,a})}

\newcommand{\goin}[2]{#1_#2}    
\newcommand{\goi}[2]{=}    
\newcommand{\Hom}{\mathrm{Hom}}    
\newcommand{\rr}{\overset{=}{\mapsto}}    
\newcommand{\Ih}{\hat I}    
\newcommand{\Ihb}{\hat I_\bl}    
\newcommand{\Ihw}{\hat I_\wh}    
\newcommand{\y}[1]{\mathscr Y_{#1}}    
\newcommand{\aaa}[1]{\mathscr A_{#1}}

\newcommand{\on}{\triangleright}    
\newcommand{\rep}[1]{{\mathscr{R}\! ep(#1)}}    
\newcommand{\groth}[1]{{\mathrm{Rep}(#1)}}    
\newcommand{\Cx}{{\mathbb C_{\!\neq 0}}}    
    
\newcommand{\btp}{\begin{tikzpicture}[baseline=0pt,scale=0.9,line width=0.25pt]}    
\newcommand{\etp}{\end{tikzpicture}}    
\advance\textheight by \topskip    
\begin{document}    
    
\baselineskip 17.5pt    
\parindent 18pt    
\parskip 8pt

\begin{titlepage}    
\begin{flushright}    
{\bf \today} \\    
YITP-09-64\\    
DCPT-09/71\\  
LPT-09-83  
\end{flushright}    
\begin{centering}    
\vspace{.2in}

{\Large {\bf     
Dorey's Rule and the q-Characters of     
    
Simply-Laced Quantum Affine Algebras    
}}    
    
\vspace{.3in}    

{\large C. A. S. Young ${}^{1}$ and R. Zegers ${}^2$}\\    
\vspace{.2 in}    
${}^{1}$ {\emph{Yukawa Institute for Theoretical Physics\\  
Kyoto University, Kyoto, 606-8502, Japan}}\\    
\vspace{.2 in}    
${}^{2}$ {\emph{Laboratoire de Physique Th\'eorique\\ Universit\'e Paris-Sud 11/CNRS\\ 91405 Orsay Cedex, France}}    
    
%
\footnotetext[1]{{\tt charlesyoung@cantab.net,}\quad ${}^{2}${\tt robin.zegers@th.u-psud.fr}}    
\vspace{.5in}    
    
{\bf Abstract}    
    
\vspace{.1in}

\end{centering}

{Let $\uqgh$ be the quantum affine algebra associated to a simply-laced simple Lie algebra $\g$.    
We examine the relationship between Dorey's rule, which is a geometrical statement about Coxeter orbits of $\g$-weights, and the structure of $q$-characters of fundamental representations $V_{i,a}$ of $\uqgh$. In particular, we prove, without recourse to the ADE classification, that the rule provides a necessary and sufficient condition for the monomial $1$ to appear in the $q$-character of a three-fold tensor product $V_{i,a}\otimes V_{j,b} \otimes V_{k,c}$.}

\end{titlepage}

\begin{flushright}    
\break

\end{flushright}    
\vspace{1cm}    
    
\section{Introduction}    
\subsection{Background}    
This paper concerns the relationship between the representation theory of simply-laced quantum affine algebras on the one hand, and, on the other, the particle fusing rule originally given by Dorey in the context of affine Toda field theories.     
    
Recall that Affine Toda Field Theories (ATFTs) are integrable quantum field theories in 1+1 dimensions \cite{CorriganReview}. Let $\g$ be any simply-laced simple Lie algebra, and $I$ the set of nodes of the Dynkin diagram of $\g$. The (real coupling, purely elastic) ATFT associated to the untwisted affine algebra $\widehat\g$ has $\rank \,\g\,$ species of particles, labelled by the nodes $i\in I$. The root system data of $\g$ determine not only the masses of these particles, but also the allowed \emph{fusings}: if particles of species $j\in I$ and $i\in I$ can interact to form a particle of species $\bar k\in I$ one says there is a fusing $j,i \rightarrow \bar k$, and this process can occur only if the rapidities $\theta_i,\theta_j$ of the incoming particles are related by    
\be \theta_i - \theta_j = \sqrt{-1}\,\, \theta^k_{ji} \ee     
where $\theta^k_{ji}$ is a real angle, called the \emph{fusing angle}. If there is a fusing $j,i\rightarrow \bar k$ then there are also fusings $i,k\rightarrow \bar\jmath$ and $k,j \rightarrow \bar\imath$, and the fusing angles obey    
\be \theta^k_{ji} + \theta^j_{ik} + \theta^i_{kj} = 2\pi.\ee     
    
The problem of determining the masses, fusings and fusing angles for the ATFTs associated to all simple Lie algebras (simply-laced or not) was solved in \cite{BCDS}. It was observed in that paper that the allowed fusings form a strict subset of the non-zero Clebsh-Gordon coefficients for $\g$, in the sense that if $i,j\rightarrow \bar k$ is a fusing then     
\be \Hom_{\g} \left( V_i \otimes V_j, V_{\bar k} \right) \cong \Hom_{\g}\left( V_i \otimes V_j \otimes V_k, \mathbb C\right) \neq 0,\label{ghom}\ee    
where $V_i$ is the $i$th fundamental representation of $\g$. It is a \emph{strict} subset because the converse statement does not hold: the first counterexample is $D_5$,     
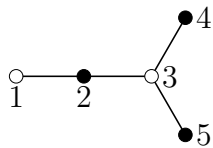
\be\nn\btp \draw[draw=white,double=black,very thick] (1,0) -- ++(2,0) -- ++(60:1) ++(60:-1) -- ++(-60:1) ;    
\filldraw[fill=white] (1,0) circle (1mm) node [below] {$1$};    
\filldraw[fill=white] (3,0) circle (1mm) node [right] {$3$};    
\filldraw[fill=black] (2,0) circle (1mm)  node [below] {$2$};    
\filldraw[fill=black] (3,0)++(60:1) circle (1mm)  node [right] {$4$};;    
\filldraw[fill=black] (3,0)++(-60:1) circle (1mm) node [right] {$5$};;    
\etp\ee    
where there is a non-trivial homomorphism $V_2 \otimes V_2 \rightarrow V_2$ of $\mathfrak d_5$ modules but no fusing $2,2\rightarrow 2$ in the ATFT. Soon after, this same ``hole'' in the allowed interactions was also found in a different (and non-diagonal) scattering theory \cite{M1}, giving an indication of a more general underlying structure.    
    
Subsequently, Dorey gave a rule which encodes both the pattern of allowed fusings, and the fusing angles, in an elegant geometrical fashion for all the simply-laced cases \cite{D,DII,Dorey:1992gr,FringLiaoOlive,FringOlive}.\footnote{A generalization of the rule to non-simply laced cases was mentioned in \cite{Dorey:1993tn} and used in \cite{CP95}; see also \cite{Oota,FKS}. In the present work we shall focus exclusively on the simply laced cases but it would be very interesting to try to prove analogous results for any simple Lie algebra.} To state the rule, we introduce some standard notation: let $(\alpha_i)_{i\in I}$ be a set of simple roots of $\g$, $(\lambda_i)_{i\in I}$ the corresponding fundamental weights, and $a_{ij}$ the Cartan matrix:    
\be \alpha_i \cdot \alpha_j = a_{ij},\qquad \alpha_i \cdot \lambda_j =    
\delta_{ij} .\ee    
Let $W$ denote the Weyl group of $\g$, generated by the reflections $(s_i)_{i\in I}$ in the simple roots. It is always possible to write $I$ as a disjoint union    
\be I= I_\bl \disjoint I_\wh\ee    
in such a way that $(I_\bl,I_\wh)$ is a two-colouring of the Dynkin diagram (as, for example, in the case $D_5$ above). Let then $w\in W$ be the choice of \emph{Coxeter element} given by\footnote{This choice will be convenient in what follows, but the rule itself is independent of the choice of Coxeter element.}    
\be w = w_\wh w_\bl, \qquad w_\wh = \prod_{i \in I_\wh} s_i,     
                     \qquad w_\bl = \prod_{i\in I_\bl} s_i,\label{wdef}\ee    
and write $\Gamma = \langle w\rangle$ for the cyclic subgroup of $W$ generated by $w$, whose order $h$ is the \emph{Coxeter number} of $\g$.     
    
Then the rule states that there is a fusing $i,j\rightarrow \bar k$ if and only if    
\be 0 \in  \Gamma \lambda_i + \Gamma \lambda_j + \Gamma \lambda_k ;\label{Dorey}\ee    
that is, if and only if there are integers $p,q,r$ such that     
\be 0 = w^p \lambda_i + w^q \lambda_j + w^r \lambda_k.\ee     
Moreover, the fusing angles, $\theta_{ij}^k$  $\theta_{jk}^i$ and  $\theta_{ki}^j$ are given by projecting this latter equation onto the $\exp\left(\pm 2\pi i/h\right)$ eigenplane of $w$, as discussed in \cite{D,DII} and recalled in section \ref{fuzsec} below. The original statement of the rule involved Coxeter orbits of \emph{roots}, but it was observed in \cite{Braden} that the statement above in terms of weights is equivalent, essentially because (one can show that) $\phi_i:= (1-w^{-1})\lambda_i$ are a linearly independent collection of roots. Writing the rule in terms of weights is suggestive, because of the following    
\begin{theorem}[PRV \cite{PRV,Kumar,Mathieu}] \label{PRVthm} A necessary and sufficient condition for    
\be \Hom_{\g}\left( V_i \otimes V_j \otimes V_k, \mathbb C\right) \neq 0 \ee      
is that    
\be 0 \in W \lambda_i + W\lambda_j + W \lambda_k .\label{PRV}\ee     
\end{theorem}    
Now clearly (\ref{Dorey}) implies (\ref{PRV}), but not vice versa. So in light of this result, which connects the Weyl-orbits of weights to invariants of $\g$-representations, it is very natural to suppose that the fusing rule (\ref{Dorey}) plays a similar role for representations of some larger (and hence more restrictive) algebraic structure. In \cite{M2}, MacKay conjectured that this is indeed the case and that the relevant algebra is the Yangian $\Yg$.     
    
Recall that  the universal envelope $U(\widehat\g)$ of the untwisted affine algebra $\widehat\g$ has a canonical Drinfel'd-Jimbo deformation $\uqgh$, called a \emph{quantum affine algebra}, and that the \emph{Yangian} $\Yg$ is the \emph{rational degeneration} of $\uqgh$ \cite{DrinfeldYBE,Drinfeld}. $\Yg$ and $\uqgh$ share essentially the same representation theory \cite{Varagnolo}. There is a notion of the fundamental representations $V_{i,a}$ of $\uqgh$, where $i\in I$, and $a\in \Cx$ is an additional label which we will call the  rapidity; see e.g. \cite{CPbook}.\footnote{We are of course using  the word ``rapidity'' in two, a priori different, senses: for the kinematical label of particles in ATFT and for the spectral parameter of representations of $\uqgh$. The role of $\uqgh$-symmetry in real- and imaginary-coupling affine Toda field theory is indeed rather subtle. See \cite{TW,SK} and references therein.    
}    
The $V_{i,a}$ are finite-dimensional and $V_i\subset V_{i,a}$.    
    
In the classical cases, the following theorem was proved by Chari and Pressley, confirming the conjecture above. (In fact, \cite{CP95} considered all the classical cases $ABCD$, but we quote here only the result for the classical simply-laced cases $AD$.)       
\begin{theorem}[\cite{CP95}] A necessary and sufficient condition for    
\be \Hom_{\uqgh}\left( V_{i,a} \otimes V_{j,b} \otimes V_{k,b}, \mathbb C\right) \neq 0 \, ,\ee      
for some rapidities $a,b,c\in \Cx$, is that    
\be 0 \in  \Gamma \lambda_i + \Gamma \lambda_j + \Gamma \lambda_k.\ee\label{CPthm}    
\end{theorem}    
    
\subsection{Motivations and Outline}    
Despite the positive result above, it is fair to say that a satisfactory understanding of the link between the fusing rule and the representation theory of simply-laced quantum affine algebras is still missing. Most apparently, the proof in \cite{CP95} was case-by-case and did not include the exceptional cases $E_6$, $E_7$ and $E_8$. More importantly, part of what makes the rule (\ref{Dorey}) elegant is that it encodes not only the triples $(i,j,k)$ for which fusing can occur, but also the fusing angles, via the projection map mentioned above. This aspect played no role in \cite{CP95}, where the required rapidities were determined without reference to this projection map.     
One would like to understand \emph{why} the rapidities emerge as they do from the geometry of Coxeter orbits of roots and weights.    
    
In the present paper we take a step in this direction, by relating the geometry of Coxeter orbits to the \emph{$q$-characters} of fundamental representations of $\uqgh$. The notion of $q$-characters, due to  Frenkel and Reshetikhin \cite{FR}, 
following \cite{Knight}, is an important development in the representation theory of quantum affine algebras. 
Here they will allow us to give, in particular, a general proof that Dorey's rule is a \emph{necessary} condition for the existence of invariant maps, $\Hom_{\uqgh}\left( V_{i,a} \otimes V_{j,b} \otimes V_{k,c}, \mathbb   C\right) \neq 0$, and singlets, $\Hom_{\uqgh} (\mathbb C,  V_{k,c} \otimes V_{j,b} \otimes V_{i,a}) \neq 0$. 

The structure of this paper is as follows: in section \ref{qcharsec} we recall the definition of $\uqgh$, and the necessary details of the theory of $q$-characters. 

Then in section \ref{fuzsec} we go on to prove our main result (theorem \ref{sec3thm}), which states that Dorey's rule provides a necessary and sufficient condition for the monomial $1$ to occur in the $q$-character of a three-fold tensor product of fundamental representations. We prove this by first showing (lemma \ref{qmonlemma}) that the latter statement can be rephrased as a statement about the occurrence of quadratic monomials in the $q$-character of a single fundamental representation. We then prove that such quadratic monomials are in a certain precise correspondence with solutions to Dorey's rule. 

Indeed, it will emerge that in fact \emph{every} monomial in the $q$-character can very naturally be seen as specifying some identity among the Coxeter orbits of the fundamental weights of $\g$ (proposition \ref{1impliesrule}). The reverse direction however (going from identities to monomials) is more subtle, and one must work harder to show (propositions \ref{routeprop} and \ref{inchiq}) that it always holds for identities of the form (\ref{Dorey}) above.    

We conclude in section \ref{outlook} by commenting on the relationship of our result to the theorem \ref{CPthm} above, and noting some open questions.    
    
We assume, throughout this paper, that $q \in \Cx$ is not a root of unity.

\section{Quantum Affine Algebras and $q$-characters}\label{qcharsec}    
The \emph{quantum affine algebra} $\uqgh$ is an associative algebra over $\mathbb C$     
generated by \be (x^\pm_{i,n})_{i \in I, n \in \mathbb Z}, \quad\quad (k_i^{\pm 1})_{i \in I}, \quad\quad  (h_{i, n})_{i \in I , n \in \mathbb Z_{\!\neq 0}},\label{gens}\ee     
and central elements $c^{\pm 1/2}$. In this paper we study finite dimensional representations of $\uqgh$ when $\g$ is simply laced. As we recall below, for this purpose it actually suffices to work with the \emph{quantum loop algebra} $U_q(L\g) = \uqgh/(c^{\pm 1/2}-1)$.     
Following \cite{Drinfeld}, let us arrange the generators into formal series    
\be x^\pm_i(u) := \sum_{n \in \mathbb Z} x^{\pm}_{i,n} u^{-n} \label{xiu} \ee    
\be \phi_i^\pm (u) = \sum_{n=0}^\infty  \phi_{i,\pm n}^\pm u^{\pm n} :=     
k_i^{\pm 1} \exp \left(\pm (q - q^{-1}) \sum_{m=1}^\infty h_{i,\pm m} u^{\pm m} \right ) \, , \label{phiu} \ee    
and set    
\be \delta(u) := \sum_{n \in \mathbb Z} u^n .\ee    
The defining relations of $U_q(L\g)$ are then    
\be    
\left [\phi^\pm_i(u) , \phi^\pm_j (v) \right ] = \left [ \phi^\pm_i (u) , \phi^\mp_j(v) \right ] = 0 \ee    
\be    
\phi^\pm_i(u)\,  x^+_j(v) = q^{a_{ij}} \, \frac{1- q^{-a_{ij}} uv}{1- q^{a_{ij}}uv} \, x^+_j(v) \, \phi^\pm_i(u) \label{phix+}\ee    
\be    
\phi^\pm_i(u) \, x^-_j(v) = q^{-a_{ij}} \, \frac{1- q^{a_{ij}} uv}{1- q^{-a_{ij}}uv} \, x^-_j(v) \, \phi^\pm_i(u) \label{phix-}    
\ee    
\be \left [x^+_i(u), x^-_j(v) \right ] = \frac{\delta_{ij}}{q-q^{-1}} \left (\delta(v/u) \phi^+_i(1/v) - \delta(u/v) \phi^-_i(1/u) \right ) \label{x+x-}\ee    
\be \left ( u-q^{\pm a_{ij}}v \right )x^\pm_i(u) \, x^\pm_j(v) =     
\left (q^{\pm a_{ij}}u-v \right)\, x^\pm_j(v) \, x^\pm_i(u) \label{x+x+1}\ee    
\bea && x^\pm_i(u) x^\pm_i(v) x^\pm_j(w) - (q+q^{-1}) x^\pm_i(u) x^\pm_j(v) x^\pm_i(w) \nonumber \\    
 && \qquad + \,\,  x^\pm_j(v) x^\pm_i(u) x^\pm_i(w) \quad + \quad \left( u \leftrightarrow v \right)= 0 \qquad \qquad \mbox{if $a_{ij}=-1$,}\eea    
where $a_{ij}$ is the Cartan matrix of $\g$.    
As we shall see, this presentation, which is a slightly modified version of Drinfel'd's \emph{current} presentation \cite{Drinfeld}, is convenient because the $\phi_i^\pm(u)$ and $x^\pm_i(u)$ behave analogously to the usual Cartan generators and raising/lowering operators in the representation theory of finite-dimensional simple Lie algebras. From its origin as a standard Drinfel'd-Jimbo deformation of $U(\widehat \g)$, $\uqgh$ admits a \emph{standard} Hopf algebra structure $\uqgh^{\mbox{\tiny std}}$-- see e.g. \cite{CPbook}. No closed form is known for the standard coproduct in the current presentation above. As we note in the conclusion, there does exist another (twist-equivalent \cite{EKP}) Hopf algebra structure for $\uqgh$ better suited to the current presentation; for details see \cite{HernandezFusionI,HernandezFusionII,Grosse}.    
    
\subsection{Representations and Characters}    
    
A representation $V$ of $\uqgh$ is \emph{of type $1$} if $c^{\pm 1/2}$ acts as the identity on $V$ and $V$ is the direct sum of its $U_q(\g)$-weight spaces,     
\be V = \oplus_\lambda V_\lambda\quad\text{ where}\quad     
 V_\lambda = \{ v \in V: k_i \on v = q^{\langle \alpha_i, \lambda \rangle} v\}\ee
and $\lambda$ in the weight lattice of $\g$.     
We recall (see e.g. \cite{CPbook} chapter 12.2B) that any finite-dimensional irreducible representation of $\uqgh$ can be obtained by twisting, by an automorphism of $\uqgh$, a finite-dimensional type 1 representation. Thus it suffices for our purposes to consider type 1 representations, and to regard them as representations of $U_q(L\g)$.    
    
Any type $1$ representation $V$ of $\uqgh$ also furnishes a representation of $\uqg$ (the latter being  the subalgebra generated by $(x^\pm_{i,0})_{i\in I}, (k^\pm_i)_{i\in I}$). Recall that the \emph{character} $\chi(V)$ of $V$ regarded as a $\uqg$-module is defined as    
\be \chi(V) = \sum_\lambda \dim \left(V_\lambda\right)  e^\lambda.\ee    
If $\rep\uqg$ is the category whose objects are finite-dimensional representations of $\uqg$ and whose morphisms are homomorphisms of $\uqg$-modules, then the \emph{Grothendieck ring} $\groth\uqg$ is the ring generated by the isomorphism classes of objects in $\rep\uqg$ subject to the relations $[X][Y]= [X\otimes Y]$ and, for each exact sequence $0\to U \to W\to V \to 0$  of $\uqg$-modules,  $[W] = [U] + [V]$.    
The character map $\chi$ is a homomorphism of rings    
\be \chi : \groth\uqg \longrightarrow     
                                  \mathbb Z\left[y_i^{\pm 1} \right]_{i \in I}\ee    
to the ring of polynomials in variables $y^{\pm 1} _i= e^{\pm \lambda_i}$.

Let us pause to recall that $\rep\uqg$, like $\rep{U(\g)}$, is a \emph{semisimple} category: exact sequences $0\to U \to W\to V \to 0$ exist precisely when $W= U \oplus V$ as $\uqg$-modules; and thus the defining relations of $\groth\uqg$ are in fact just $[U][V] = [U\otimes V]$ and $[U] + [V] = [U\oplus V]$. 
In contrast, representations of $\uqgh$ can be reducible but not fully-reducible. That is, it can happen that there is a short exact sequence $0 \to U \to W \to V \to 0$ of $\uqgh$-modules, so that $U$ is a submodule of $W$, but that $W$ is not the direct sum $U\oplus V$ as a $\uqgh$-module. One says that $W$ is \emph{indecomposable}.    
    
Now for any type 1 representation $V$ of $\uqgh$, the decomposition above into $\uqg$-weight spaces may be further refined by decomposing $V$ into Jordan subspaces of the mutually commuting $\phi_{i,\pm r}^\pm$ defined in (\ref{phiu}), \cite{FR}:    
\be V = \oplus_{\pmb\gamma} V_{\pmb\gamma}, \quad \pmb\gamma = (\gamma_{i,\pm r}^\pm)_{i\in I, r\in \mathbb N}, \quad \gamma_{i,\pm r}^\pm \in \mathbb C\ee    
where 
\be V_{\pmb\gamma} = \{ v \in V : \exists N \in \mathbb N, \,\, \forall i \in I, \quad \left( \phi_{i}^\pm(u) - \gamma_{i}^\pm(u)\right)^N \on v = 0 \} \, . \label{lweightspacesdef}\ee    
If $\dim (V_{\pmb\gamma}) >0$, we shall refer to the corresponding formal series    
\be\forall i \in I\, , \qquad  \gamma_i^\pm(u) := \sum_{r \in \mathbb N}  \gamma_{i,\pm r}^\pm u^{\pm r} \ee    
as an \emph{l-weight} of $V$. It is known \cite{FR} that for every finite-dimensional type 1 representation of $\uqgh$, these l-weights are of the form    
\be \gamma_i^\pm(u) = q^{\deg Q_i - \deg R_i} \, \frac{Q_i(uq^{-1}) R_i(uq)}{Q_i(uq) R_i(uq^{-1})} \, ,\ee    
where the right hand side is to be treated as a formal series in positive (negative) integer powers of $u$ for $\gamma^+_i(u)$ (respectively $\gamma^-_i(u)$), and $Q_i$ and $R_i$ are polynomials with constant term $1$. These latter may be written as    
\be Q_i(u) = \prod_{a\in \Cx} \left( 1- ua\right)^{q_{i,a}}, \quad\quad     
    R_i(u) = \prod_{a\in \Cx} \left( 1- ua\right)^{r_{i,a}}, \ee    
and this allows one to assign to $\pmb\gamma$ a monomial    
\be m_{\pmb\gamma} = \prod_{i \in I, a\in \Cx} Y_{i,a}^{q_{i,a}-r_{i,a}}\ee    
in variables $(Y_{i,a})_{i\in I; a\in \Cx}$. The $q$-character map $\chi_q$ \cite{FR} is the injective homomorphism of rings    
\be \chi_q : \groth\uqgh \longrightarrow \mathbb Z\left[ Y_{i,a}^{\pm 1} \right]_{i\in I, a \in \Cx} \label{chiqintro}\ee    
defined by\footnote{Note that the original definition of $\chi_q$ \cite{FR} was in terms of the universal $\mathcal R$-matrix of $\uqgh$, which makes its close relationship to the transfer matrix of physics more evident. But the above definition, c.f. e.g. \cite{CH}, is more directly suited for our purposes.}    
\be \chi_q (V) = \sum_{\pmb\gamma} \dim\left(V_{\pmb \gamma}\right) m_{\pmb\gamma}.\ee    
The $Y^{\pm 1}_{i,a}$ are to be thought of as the quantum-affine analogues of the usual variables $y^{\pm 1}_j = e^{\pm \lambda_j}$ appearing in character polynomials. In particular, one associates $Y_{i,a}^{\pm 1}$ with the classical weight $\pm \lambda_i$. An algorithm for computing $q$-characters of $\uqgh$-modules directly from the root-system data of $\g$ was proposed in \cite{FR,FM}. It has been proven to work for all fundamental representations \cite{FM}, which is all that we shall require in the present paper, although it is known not to work in general \cite{HernandezLeclerc,Nakai}. In \cite{Nakajima3}, Nakajima deduced an algorithm for computing the $q$-character of any irreducible representation, and formulas for the $q$-characters of fundamental representations were given in \cite{NakajimaAD,NakajimaE}; see also \cite{CM}.

We now turn to summarizing the properties of $q$-characters that we shall need.    
A monomial in $\mathbb Z\left[ Y^{\pm 1}_{i,a}\right]_{i\in I;\, a\in \Cx}$ is said to be \emph{$i$-dominant} if and only if it contains no $Y_{i,a}^{-1}$'s. It is said to be \emph{dominant} if and only if it is $i$-dominant for all $i\in I$. Antidominant monomials are similarly defined to be those not containing $Y_{i,a}$'s.       
     
\subsection{$\uqslt$ characters}\label{sl2chars}    
We first summarize the situation for $\uqslt$ characters. In this case the Dynkin diagram has one node, $I= \{1\}$, and we write $Y_{1,a}^{\pm 1} = Y_a^{\pm 1}$. The fundamental representations $V_a$ of $\uqslt$ have dimension two and are labelled by the rapidity $a\in\Cx$. Their $q$-characters are    
\be \chi_q(V_a) = Y_a \left( 1+A_{aq}^{-1}\right)     
            = Y_a + Y_{aq^2}^{-1},\ee     
where one defines    
\be A_{a} = Y_{aq} Y_{aq^{-1}}.\ee    
The tensor product $V_b\otimes V_c$ of two fundamental representations is irreducible whenever $b/c \notin \{ q^{- 2},q^{+2}\}$. When $b=aq$ and $c= aq^{-1}$ for some $a\in \Cx$, there is an exact sequence of $\uqslt$-modules (\cite{CPsl2}, and with their choice of coproduct)    
\be 0 \rightarrow W^{(2)}_a \rightarrow V_{aq}\otimes V_{aq^{-1}} \rightarrow \mathbb C \rightarrow 0 \ee    
where $W_a^{(2)}$ is a 3-dimensional irreducible submodule and $\mathbb C \cong  \left(V_{aq}\otimes V_{aq^{-1}}\right) \big/ W^{(2)}_a$ is the 1-dimensional module. If instead $b=aq^{-1}$ and $c=aq$, one has the same exact sequence but with arrows reversed:    
\be 0 \rightarrow \mathbb C \rightarrow V_{aq^{-1}}\otimes V_{aq} \rightarrow W^{(2)}_a \rightarrow 0. \ee    
In either case, there is more than one dominant monomial in the $q$-character:    
\bea \chi_q(V_{aq^{-1}}\otimes V_{aq}) =  \chi_q(V_{aq}\otimes V_{aq^{-1}}) &=& \chi_q(V_{aq^{-1}}) \chi_q(V_{aq}) \label{mt1}\\     
   &=& \left(Y_{aq^{-1}}+ Y_{aq}^{-1}\right) \left( Y_{aq} + Y_{aq^3}^{-1}\right)\nn\\    
    &=& 1 +  \left( Y_{aq^{-1}} Y_{aq} + Y_{aq^{-1}} Y_{aq^3}^{-1} + Y_{aq}^{-1} Y_{aq^3}^{-1} \right) \nn.\eea    
In the final line the quantity in brackets is $\chi_q(W^{(2)}_a)$. More generally, for each $r\in \mathbb Z_{\geq 1}$ and $a\in \Cx$ there is an irreducible submodule    
\be  W^{(r)}_a\subset V_{aq^{r-1}} \otimes V_{aq^{r-3}} \otimes \dots \otimes V_{aq^{-r+1}} \ee    
called the $r$-th Kirillov-Reshetikhin module of $\uqslt$. It has dimension $r+1$ and $q$-character\footnote{$W_a^{(r)}$ is the pull-back of the usual spin $r/2$ representation of $U_q(\mathfrak{sl}_2)$ by the \emph{evaluation homomorphism} $\mathrm{ev}_a:\uqslt \to U_q(\mathfrak{sl}_2)$. See e.g \cite{CPsl2}.}    
\bea \chi_q(W^{(r)}_a) &=& \left(Y_{aq^{-r+1}} Y_{aq^{-r+3}} \dots Y_{aq^{r-1}}\right)      
                      \left(1+ \sum_{t=0}^{r-1}     
             A^{-1}_{aq^r} A^{-1}_{aq^{r-2}} \dots A^{-1}_{aq^{r-2t}}\right) \label{chiW}\\    
                  &=&\phantom{{}+{} }     
       Y_{aq^{-r+1}} Y_{aq^{-r+3}} \dots Y_{aq^{r-3}}          Y_{aq^{r-1}} \nn\\    
& &{}+ Y_{aq^{-r+1}} Y_{aq^{-r+3}} \dots Y_{aq^{r-3}} \phantom{Y_{aq^{1+1}}} Y_{aq^{r+1}}^{-1} \nn\\    
& & \qquad\qquad\qquad\qquad\dots\nn\\    
&& {}+ Y_{aq^{-r+1}} \phantom{Y_{aq^{1+1}}} Y_{aq^{-r+5}}^{-1} \dots Y_{aq^{r-1}}^{-1} Y_{aq^{r+1}}^{-1} \nn\\    
&& {} + \phantom{Y_{aq^{1+1}}} Y_{aq^{-r+3}}^{-1} Y_{aq^{-r+5}}^{-1} \dots Y_{aq^{r-1}}^{-1} Y_{aq^{r+1}}^{-1}.\nn\eea      
$W_a^{(r)}$ is completely characterised by the set of rapidities $S_{r}(a) = \left \{a q^{-r+1}, aq^{-r+3}, \dots, aq^{r-1} \right \}$ appearing in its dominant monomial, which we shall refer to as a \emph{segment} of length $r$ centred on $a$. Two such segments are said to be in \emph{special position} if their union is itself a segment and neither of them contains the other. We say $aq^{-r+1}$ is the \emph{leftmost} element of $S_{r}(a)$, $aq^{r-1}$ the \emph{rightmost}. More generally we say that $aq^k$ is \emph{to the right} (\emph{left}) of $aq^l$ iff $k>l$ (resp. $k<l$). 
    
Presented with any dominant monomial $m_+=\prod_s Y_{a_s}$ one can reconstruct the unique irreducible $\uqslt$-module $V(m_+)$ such that $m_+$ is the \emph{highest weight monomial} in $\chi_q(V(m_+))$.     
First split the factors $Y_{a_s}$ into a product of segments no two of which are in special position: say    
\be m_+ = \prod_{t\in T} \left(Y_{a_tq^{-r_t+1}} Y_{a_tq^{-r_t+3}} \dots Y_{a_tq^{r_t-1}}\right),\ee    
for some index set $T$; then     
\be V(m_+) \cong \bigotimes_{t\in T} W^{(r_t)}_{a_t}\label{Vm+},\ee 
which can be shown to be irreducible and, up to isomorphism, independent of the ordering of the tensor factors.   
    
Finally, there is an important caveat: reducible modules certainly have more than one dominant monomial, as in e.g.  (\ref{mt1}), but irreducible modules can \emph{also} have multiple dominant monomials. This happens precisely when they fail to be \emph{regular}, in the terminology of \cite{FR}. Consider $m_+ = Y_{aq^{-1}}^2 Y_{aq}$ to see the problem. Note that the resulting $q$-character contains the (dominant) monomial $Y_{aq^{-1}}$ but \emph{not} the monomial $Y_{aq}^{-1}$. Thus, in computing $q$-characters, one cannot treat all dominant monomials as though they were highest monomials. For that reason we shall need the following
\begin{proposition}
\label{yan}
Let $V$ be a simple finite dimensional $\uqslt$-module of type 1. Suppose that for some $a\in\Cx$ and $n>0$, $\chi_q(V)$ includes a dominant  monomial $m$ such that $Y_{a}^n$ is a factor\footnote{For every $b\in \Cx, k\in \mathbb Z_{\neq 0}$, we say that $Y_{b}^k$ is a factor of $m=\prod_{c\in\Cx} Y_c^{u_c}$ iff either $u_b \geq k >0$ or $u_b \leq k <0$.}
 of $m$ and $Y_{aq^2}$ is not. Then, either
\begin{enumerate}[i)]
\item $\chi_q(V)$ includes the monomials $m A_{aq}^{-p}$, $1\leq p\leq n$; or
\item there exists a $k>0$ such that $\chi_q(V)$ includes the monomial $m A_{aq^k}$.
\end{enumerate}
\end{proposition}
\begin{proof}
Let $T$ be an index set such that $V$ can be written as in (\ref{Vm+}) above, with the $S_{r_t}(a_t)$ in pairwise general position and $r_t >0$ for all $t \in T$. By hypothesis there exist $(m_t)_{t \in T}$ such that $m_t$ is a monomial of $\chi_q(W^{(r_t)}_{a_t})$ for each $t\in T$ and
\be \label{yanprod} m = \prod_{t \in T} m_t \, .\ee 
Let $T'=\{t\in T: a\in S_{r_t}(a_t)\}$. 
Note that $Y_{aq^2}^{-1}$ is a factor of $m_t$ only if  $m_t\in T'$.  
$T'$ is the disjoint union of the following three subsets: 
\bea\nn T'_1 &=& \{t\in T': \text{$m_t$ is dominant and has both $Y_a$ and $Y_{aq^2}$ as factors}\},\\ 
\nn T'_2 &=& \{t\in T': \text{$m_t$ is dominant and has rightmost factor $Y_{a}$}\},\\
\nn T'_3 &=& \{t\in T': \text{$m_t$ is not dominant}\}.\eea
If there is a $t\in T'_3$ such that $Y_a^{-1}$ is not a factor of $m_t$ then the leftmost factor $Y^{-1}$ in $m_t$ is $Y_{aq^{2\ell}}^{-1}$ for some $\ell >0$. In that case $m_tA_{aq^{2\ell-1}}$ appears in $\chi_q(W^{(r_t)}_{a_t})$, c.f. (\ref{chiW}), and ii) holds. 

It remains to consider the case that $Y_a^{-1}$ is a factor of every $m_t$, $t\in T'_3$. By definition of $T'$, $Y_{aq^2}^{-1}$ is then also a factor of every $m_t$, $t\in T'_3$. 
Suppose for a contradiction that there existed a $t\in T$ such that $m_t$ is dominant with leftmost factor $Y_{aq^2}$. Since by assumption the total power of $Y_{aq^2}$ in $m$ is zero, that would require $|T'_1|<|T'_3|$; but also, by definition of general position, that $|T'_2|=0$ and hence $|T'_1|-|T'_3|\geq n>0$, a contradiction. Therefore there is no such $t\in T$ and so in fact, by counting powers of $Y_{aq^2}$ in $m$, $|T'_1|=|T'_3|$. Consequently the power of $Y_a$ in $\prod_{t\in T'_1\sqcup T'_3} m_t$ is zero. It follows that $|T'_2|\geq n$ and hence that i) holds. 
\finproof
\end{proof}


\subsection{$\uqgh$ characters}    
\label{sec:uqgchar}    
Returning to the general case, we let $V_{i,a}$, $i\in I, a\in \Cx$ denote the $i$-th fundamental representation of $U_q(\widehat\g)$ at rapidity $a$. (See e.g. \cite{CPbook}.) It may be shown \cite{FR, FM} that $\chi_q(V_{i,a})$ contains the highest weight monomial $Y_{i,a}$ and that, if we define    
\be A_{i,a} = Y_{i,aq^{-1}} Y_{i,aq} \prod_{\nbr j i } Y_{j,a}^{-1},\ee    
where the product $\prod_{\nbr j i}$ is over the nodes $j$ of the Dynkin diagram that neighbour $i$,\footnote{That is $\prod_{\nbr j i}=\prod_{j:I_{ji}=1}$ where     
\be I_{ij} = 2 \delta_{ij} - a_{ij} = \begin{cases} 1 \quad \text{if    
$i,j$ are neighbouring nodes on the Dynkin diagram}\\  0 \quad    
\text{otherwise}\end{cases}\ee    
is the incidence matrix.}    
then every monomial in $\chi_q(V_{i,a})$ is of the form    
\be Y_{i,a} A_{j_1,a_1}^{-1} \dots A_{j_n,a_n}^{-1} \label{onlyainvs}\ee    
for some finite collection of $n\geq 0$ pairs $(j_k,a_k)\in I \times \Cx$. For each $j\in I$, let $\uqsl j \subset \uqgh$ be the subalgebra generated by $x_j^\pm(u),\phi_j^\pm(u)$. Let $\chi_q^{(j)}$ be the $q$-character map of $\uqsl j$ and     
\be \beta_j : \mathbb Z \left[ Y^{\pm 1}_{i,a}\right]_{i\in I; a \in \Cx} \to \mathbb Z \left[ Y^{\pm 1}_{j,a}\right]_{a \in \Cx} \ee     
the ring homomorphism which sets to one all the $Y^{\pm 1}_{k,a}$ with $k \neq j$. Then every $\uqgh$-module $V$ is also a $\uqsl j$-module, and $\chi_q^{(j)}(V) = \beta_j \circ \chi_q(V)$. In fact, more is true: there exists \cite{FM} an \emph{injective} ring homomorphism    
\be\tau_j : Z \left[ Y^{\pm 1}_{i,a}\right]_{i\in I; a \in \Cx} \to \mathbb Z \left[ Y^{\pm 1}_{j,a}\right]_{a \in \Cx} \otimes \mathbb Z \left [ Z_{k, b}^{\pm 1}\right ]_{k \neq j; b \in \Cx}\ee    
refining $\beta_j$, where $Z_{i,a}^{\pm 1}$ are certain new formal variables, and    
\be \tau_j (\chi_q(V_{i,a})) = \sum_p \chi_q^{(j)}(V_p) \otimes N_p \, , \label{taujchi}\ee    
where the $V_p$ are $\uqsl j$-modules and the $N_p$ are monomials in $(Z_{k,b}^{\pm 1})_{k\neq j, b \in \Cx}$. Furthermore, in the diagram    
\be\label{tauj}\begin{tikzpicture}    
\matrix (m) [matrix of math nodes, row sep=3em,    
column sep=4em, text height=2ex, text depth=1ex]    
{     
\mathbb Z \left[ Y^{\pm 1}_{i,a}\right]_{i\in I; a \in \Cx}    
 & \mathbb Z \left[ Y^{\pm 1}_{j,a}\right]_{a \in \Cx} \otimes \mathbb Z \left [ Z_{k, b}^{\pm 1}\right ]_{k \neq j; b \in \Cx} \\    
\mathbb Z \left[ Y^{\pm 1}_{i,a}\right]_{i\in I; a \in \Cx}      
 & \mathbb Z \left[ Y^{\pm 1}_{j,a}\right]_{a \in \Cx} \otimes \mathbb Z \left [ Z_{k, b}^{\pm 1}\right ]_{k \neq j; b \in \Cx}\\    
};    
\path[->,font=\scriptsize]    
(m-1-1) edge node [above] {$\tau_j$} (m-1-2)    
(m-2-1) edge node [above] {$\tau_j$} (m-2-2)    
(m-1-1) edge (m-2-1)    
(m-1-2) edge (m-2-2);    
\end{tikzpicture}\ee    
let the right vertical arrow be multiplication by $\beta_j(A_{j,c}^{-1}) \otimes 1$; then the diagram commutes if and only if the left vertical arrow is multiplication by $A^{-1}_{j,c}$.

Consequently, if one has found a term $m_+ \otimes N_p$ in the r.h.s of (\ref{taujchi}), and one knows that $m_+$ is the highest weight monomial of $\chi_q^{(j)}(V_p)$, then one can construct all the remaining monomials in $\chi_q^{(j)}(V_p) \otimes N_p$ (as discussed in the previous subsection) and hence their (unique) preimages in $\chi_q(V_{i,a})$.    
Frenkel and Mukhin gave an algorithm for computing the $q$-character with a given highest monomial \cite{FM}, by repeatedly completing $\uqslt$-characters in this way. They proved that it works for any $q$-character with a unique dominant monomial (and so in particular for the $q$-characters of fundamental representations). 

The specific instance of this sort of reasoning which we will require, in proposition \ref{quadrmonom}, is the following, which follows immediately from the existence and property (\ref{tauj}) of $\tau_j$ together with proposition \ref{yan} above.
\begin{proposition}\label{yang}
Let $j\in I$, $a\in\Cx$ and $n>0$. Suppose $m$ is a $j$-dominant monomial in $\chi_q(V_{i,a})$ such that $Y_{a}^n$ is a factor of $\beta_j(m)$ and $Y_{aq^2}$ is not. Then, either
\begin{enumerate}[i)]
\item $\chi_q(V_{i,a})$ includes the monomials $m A_{j,aq}^{-p}$, $1\leq p\leq n$; or
\item there exists a $k>0$ such that $\chi_q(V_{i,a})$ includes the monomial $m A_{j,aq^k}$.
\end{enumerate}
\end{proposition} 

Also, in proposition \ref{1impliesrule} below, we will need the following consequence of the Frenkel-Mukhin algorithm.
\begin{theorem}[\cite{FM}]    
Every monomial $m'\neq Y_{i,a}$ in $\chi_q(V_{i,a})$ is of the form $m A_{j,aq^{r+1}}^{-1}$ for some $j\in I$ and some $r \in \mathbb Z$, where $m$ is a monomial in $\chi_q(V_{i,a})$ having $Y_{j,aq^r}$ as a factor.\label{lowprop}    
\end{theorem}  
Equivalently but more intuitively, every monomial apart from the highest one is obtained from some (at least one) other monomial by a ``lowering step'' consisting of a replacement of the form    
\be Y_{j,aq^r} \mapsto Y_{j,aq^r} A_{j,aq^{r+1}}^{-1} = Y_{j,aq^{r+2}}^{-1} \prod_{\nbr k j} Y_{k,aq^{r+1}}.\label{lowering}\ee    
The $q$-characters of any fundamental representation $V_{i,a}$ thus has the structure of a connected directed graph, whose nodes are the monomials and whose edges are labelled by factors $A_{i,a}^{-1}$. (An example is shown in figure \ref{charfig}.)    

Finally, in the proof of lemma \ref{qmonlemma}, we will need the following results from \cite{FM}. We shall say that a monomial $m$ has \emph{compact support of length $n$ and base $d$} if $m \in \mathbb Z [Y^{ \pm 1}_{l,dq^r}]_{l \in I, 0 \leq r \leq n}$. Combining lemma 6.1 and 6.13 of \cite{FM}, we have
\begin{lemma}
All the monomials in $\chi_q(V_{l,d})$, where $l\in I$ and  $d\in \Cx$, have compact support of length $h$ and base $d$. 
\end{lemma}
\label{lemmacompact}
Moreover, a monomial
\be m = \prod_{(l, r) \in I \times \mathbb N_0} Y_{l, dq^r}^{p_{l,r}} \, , \qquad p_{l,r} \in \mathbb Z \ee
having compact support of length $n$ and base $d$ is said to be \emph{right negative} (resp. \emph{left positive}) if, in addition, there exists a $(k,s)\in I\times \mathbb N_0$ such that $p_{k,s}<0$ (resp. $p_{k,s}>0$) and for each $(l, r)\in I\times \mathbb N_0$ such that $p_{l,r} > 0$ (resp. $p_{l,r}<0$), $r < s$ (resp. $r > s$).
\begin{lemma}
\label{lemmarightneg} For all $i\in I$, $a\in \Cx$, in the $q$-character $\chi_q(V_{i,a})$  
\begin{enumerate}[i)]
\item every monomial except for the highest weight monomial, $Y_{i,a}$, is right negative, and \label{RN}
\item every monomial except for the lowest weight monomial, $Y_{\bar{\imath},aq^{h}}^{-1}$, is left positive. \label{LP}
\end{enumerate}
\end{lemma}
\begin{proof}
Part \ref{RN}) is lemma 6.5 in \cite{FM}. Proposition 6.18 in \cite{FM} states (in the simply-laced case) that $\chi_q(V_{\bar\imath,aq^{-h}})$ and $\chi_q(V_{i,a})$ are related by exchanging $Y_{j,aq^n}^{\pm 1}\leftrightarrow Y_{j,aq^{-n}}^{\mp 1}$, for all $j\in I$, $n\in \{0,1,\dots,h\}$. This map sends right-negative monomials to left-positive monomials (and vice versa). So part \ref{RN}) for $\chi_q(V_{\bar\imath,aq^{-h}})$ implies part \ref{LP}) for $\chi_q(V_{i,a})$. \finproof
\end{proof}

\begin{corollary}
\label{corollary1fund}
Let $i \in I$ and $a \in \Cx$. The monomial $1$ does not occur in $\chi_q(V_{i,a})$.
\end{corollary}
\begin{proof}
The monomial $1$ is not right negative and $1 \neq Y_{i,a}$. Thus, by lemma \ref{lemmarightneg}, it cannot appear in $\chi_q(V_{i,a})$. \finproof
\end{proof}
    
\section{Coxeter orbits and $q$-characters}\label{fuzsec}    
In this section we relate the geometry of the Coxeter orbits of $\g$-weights to the structure of $q$-characters of fundamental representations.     
Recalling our notations for the roots, weights and Coxeter element of $\g$ from the introduction, let us begin by noting the following identities.    
Write $\lambda_i = \lambda_i^\bl$ ($\lambda_i^\wh$) when $i\in I_\bl$ (respectively $I_\wh$). Then     
\be w_\bl \lambda_i^\bl = \lambda_i^\bl - \alpha_i =    
-\lambda_i^\bl + \sum_{\nbr j i}\lambda^\wh_j, \qquad     
    w_\bl \lambda_i^\wh = \lambda_i^\wh \label{wonl}\ee    
and likewise with $\wh \leftrightarrow \bl$.    
Thus    
\be \left(1 + w^{\pm 1}\right) \lambda_i^\wb =   \sum_{\nbr j i}    
\lambda_j^\bw \label{loweringrel}    
.\ee    
We also define    
\be P = \frac{2}{h} \sum_{n\in \Zh} \cos\left(\frac{2\pi n}{h}\right) w^n \label{s1proj},\ee    
which is the orthogonal (with respect to the Killing form $\langle \cdot, \cdot \rangle$) projector from the weight lattice of $\g$ to the $\exp\left(\pm 2\pi i /h\right)$-eigenplane of $w$.\footnote{Recall that the exponents of $\g$ are by definition those integers $s\in \Zh$ for which $\exp\left(2\pi i s/h\right)$ is an eigenvalue of $w$, and that $s=\pm 1$ are always exponents.}    
Let $\theta$ be the map which returns the signed angle between the projections of two given vectors in weight space into this plane, i.e. the map defined by    
\be \cos \theta(\mu,\rho) = \frac{\langle P \mu, P\rho\rangle}{\sqrt{\langle P\mu, P\mu \rangle \langle P\rho, P \rho \rangle}}\, ; \qquad \mbox{im} (\theta) = (-\pi,\pi] \label{thetadef}\ee    
and, to fix the orientation, $\theta( \mu, w \mu) = +2\pi/h$.    
To fix a direction in the plane, let $\lambda$ be any vector in weight space such that $P\lambda \neq 0$. Our main result is then
\begin{theorem}\label{sec3thm}    
Let $i_1,i_2,i_3\in I$ and $a_1,a_2,a_3\in\mathbb C_{\!\neq 0}$. The following are equivalent:    
\begin{enumerate}[i)]    
 \item \label{thi} the $q$-character     
\be \chi_q\left( V_{i_1,a_1} \otimes V_{i_2,a_2} \otimes V_{i_3,a_3}\right) \ee    
includes the monomial $1$    
\item \label{thii} there exist $n_1,n_2,n_3\in \mathbb Z$ and $a\in \mathbb C_{\!\neq 0}$ such that    
\be w^{n_1}\lambda_{i_1} + w^{n_2} \lambda_{i_2} + w^{n_3} \lambda_{i_3} = 0\label{feqn}\ee    
and    
\be a_k = aq^{\frac{h}{\pi}\theta\left(\lambda, w^{n_k} \lambda_{i_k} \right)}, \quad k=1,2,3.\label{raps}\ee    
\end{enumerate}    
\end{theorem}

Let us illustrate this with an example in the case of $E_6$, for which the Coxeter number is $h=12$. 
We label the nodes of the Dynkin diagram as in \cite{BCDS}:    
\be\nn    
\btp     
\draw (0,0) -- (1,0) -- (4,0);    
\draw (2,0) -- (2,1);    
\filldraw[fill=white] ++(0,0) circle (1mm) node[minimum height = 7mm,below] {$\mathbf{\bar  l}$};    
\filldraw[fill=black] ++(1,0) circle (1mm) node[minimum height = 7mm,below] {$\mathbf h$};    
\filldraw[fill=white] ++(2,0) circle (1mm) node[minimum height = 7mm,below] {$\mathbf H$};    
\filldraw[fill=black] ++(3,0) circle (1mm) node[minimum height = 7mm,below] {$\mathbf{\bar  h}$};    
\filldraw[fill=white] ++(4,0) circle (1mm) node[minimum height = 7mm,below] {$     \mathbf l$};    
\filldraw[fill=black] ++(2,1) circle (1mm) node[minimum height = 7mm,right] {$\mathbf L$};    
\etp    
\ee    
(This labelling is related to the masses of the corresponding particles, $\mathbf H/\mathbf h$eavy or $\mathbf L/\mathbf l$ight), in the Toda theory.) Among the solutions to the fusing rule (tabulated in \cite{BCDS}) is     
\be    w^{-2}\lambda_{\bar{\mathbf  l}} + \lambda_{\mathbf L} + w^{5}\lambda_{\mathbf h} = 0  \ee    
whose $P$-projection may be pictured as follows.     
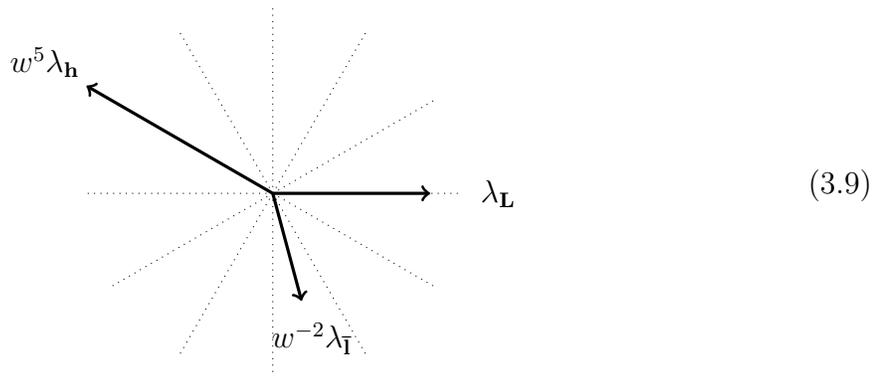
\be    
\begin{tikzpicture}[baseline=0,scale=2.5]    
\foreach \x in {0,1,2,3,4,5,6,-1,-2,-3,-4,-5}    
{    
\draw[dotted] (0,0) -- (\x*30:1);    
}    
\draw[->,very thick] (0,0) -- (0:1.2559/1.5);    
\draw[->,very thick] (0,0) -- (-5*15:0.888/1.5);    
\draw[->,very thick] (0,0) -- (10*15:1.7156/1.5);    
    
\node at (0:1.2) {$\lambda_{ \mathbf L}$};    
\node at (-5*15:.8) {$w^{-2}\lambda_{\mathbf{\bar  l}}$};    
\node at (10*15:1.4) {$w^{5}\lambda_{\mathbf h}$};    
\end{tikzpicture}     
\ee    
So the theorem 
asserts, in particular, that $1$ occurs in the $q$-character    
\be \chi_q\left(V_{\mathbf{\bar l}, aq^{-5}} \otimes V_{\mathbf L,a} \otimes V_{\mathbf h,aq^{10}}\right).\ee

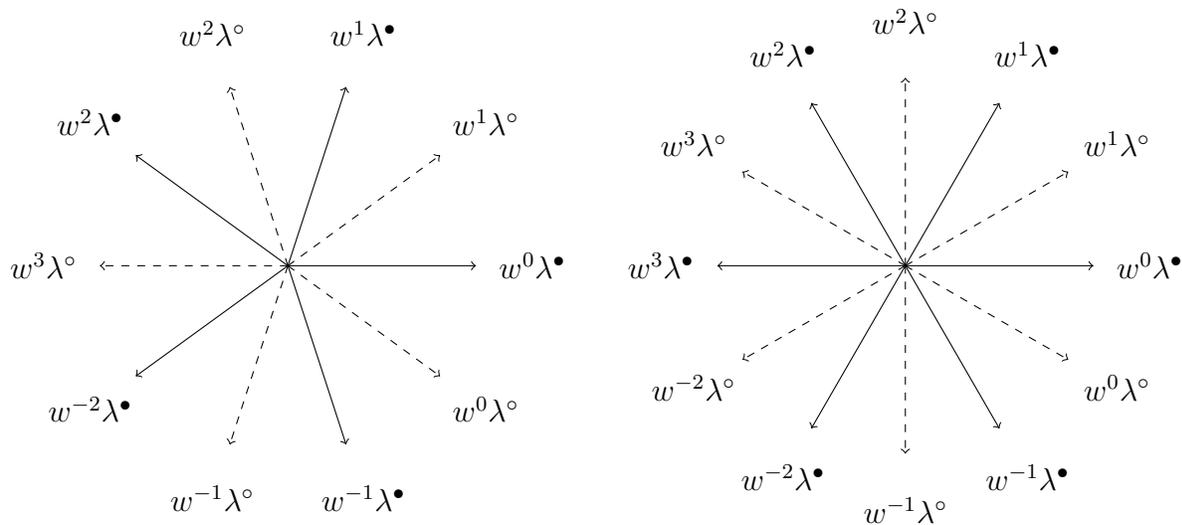
\begin{figure}    
\begin{tikzpicture}[baseline=0,scale=2.5]    
\foreach \x in {0,1,2,-1}    
{    
\draw[->] (0,0) -- (\x*72:1);    
\node at (\x*72:1.3) {$w^{\x}\lambda^\bl$};    
\draw[->,dashed] (0,0) -- ++(\x*72-36:1) ;    
\node at (\x*72-36:1.3) {$w^{\x}\lambda^\wh$};    
}    
\draw[->] (0,0) -- (-2*72:1);    
\node at (-2*72:1.3) {$w^{-2}\lambda^\bl$};    
\draw[->,dashed] (0,0) -- ++(-2*72-36:1) ;    
\node at (-2*72-36:1.3) {$w^{3}\lambda^\wh$};    
\etp    
$\quad$    
\begin{tikzpicture}[baseline=0,scale=2.5]    
\foreach \x in {0,1,2,-1,-2,3}    
{    
\draw[->] (0,0) -- (\x*60:1);    
\node at (\x*60:1.3) {$w^{\x}\lambda^\bl$};    
}    
\foreach \x in {0,1,2,-1,-2,3}    
{    
\draw[->,dashed] (0,0) -- ++(\x*60-30:1) ;    
\node at (\x*60-30:1.3) {$w^{\x}\lambda^\wh$};    
}    
\etp    
\caption{Picture of the $e^{\pm 2\pi i/h}$-eigenplane of $w$, for $h=5$ (left) and $h=6$ (right) showing the directions (though \emph{not} the lengths) of the projected Coxeter orbits of fundamental weights. Here $\lambda^\bl$ ($\lambda^\wh$) denotes any $\lambda_i$ such that $i\in I_\bl$ (respectively $I_\wh$).\label{s1pic}}     
\end{figure}    
\begin{proofof}{Theorem \ref{sec3thm}}     
We first express \ref{thii}) in a less symmetric but more convenient form. The reference vector $\lambda$ serves purely to make manifest the symmetry under permutations of $\{1,2,3\}$. It follows from (\ref{feqn}) that, by using this symmetry if necessary, we can assume    
\be -\pi < \theta(\lambda_{i_1}, w^{n_2}\lambda_{i_2}) \leq 0 <     
    \theta(\lambda_{i_1}, w^{n_3}\lambda_{i_3}) \leq \pi .\label{anglesrel}\ee    
Then, by the freedom in the choice of $a$, we can assume that $\lambda=\lambda_{i_1}$ and $n_1=0$. Let us also pick the two-colouring $I=I_\bl \sqcup I_\wh$ such that $i_1 \in I_\bl$. Given that $-\lambda_{\bar\imath_2}$ is in the Coxeter orbit of $\lambda_{i_2}$,\footnote{By definition $\lambda_{\bar\imath}$ is the fundamental weight in the Weyl orbit of $-\lambda_i$. It is given by  $\lambda_{\bar\imath} = -w_0 \lambda_i$ where $w_0$ is the longest element of the Weyl group, which may be written $w_0 = \underset{h}{\underbrace{w_\bl w_\wh \dots }} = \underset{h}{\underbrace{w_\wh w_\bl \dots}}$. Then since $w_\wh \lambda_i^\bl = \lambda_i^\bl$ and $w_\bl \lambda_i^\wh = \lambda_i^\wh$, one has $w_0\lambda_i^\bl = w^{\lfloor\frac{h}{2}\rfloor} \lambda_i^\bl$ and $w_0\lambda_i^\wh = w^{\lfloor\frac{h+1}{2}\rfloor} \lambda_i^\wh$.} we can introduce an $n\in\mathbb Z$ such that $w^{n_2}\lambda_{i_2}=-w^n\lambda_{\bar\imath_2}$. Let us also write $m:=n_3$. Then  (\ref{anglesrel}
 ) becomes
\be 0<\theta(\lambda_{i_1},w^n\lambda_{\bar\imath_2})\leq \pi,\quad 0 < \theta(\lambda_{i_1}, w^{m}\lambda_{i_3}) \leq \pi.\label{anglesrel2}\ee 
Thus the solution \ref{thii}) has been brought to the form    
\be \lambda_{i_1} - w^n \lambda_{\bar\imath_2} + w^{m} \lambda_{i_3}=0 \label{fusrew}\ee    
where, on examining figure \ref{s1pic}, one sees that (\ref{anglesrel2}) is equivalent to the following conditions on $n, m$ (modulo $h$), 
\be \label{ineqnm}    
0 < n \leq    
\begin{cases}     
     \lfloor \tfrac{h}{2} \rfloor & \bar\imath_2 \in I_{\bl}\\    
     \lfloor \tfrac{h+1}{2} \rfloor & \bar\imath_2 \in I_{\wh}    
    \end{cases}     
\qquad\qquad     
0 < m \leq \begin{cases}     
     \lfloor \tfrac{h}{2} \rfloor & i_3 \in I_{\bl}\\    
     \lfloor \tfrac{h+1}{2} \rfloor & i_3 \in I_{\wh},   
    \end{cases}     
\ee    
and that the angles in (\ref{raps}) are given by
\be a_1 = a,\quad a_2 = aq^{r-h}, \quad a_3 = aq^s \ee    
where    
\be r = \begin{cases} 2n & \bar\imath_2\in I_\bl \\ 2n-1 & \bar\imath_2 \in I_\wh    
\end{cases}\qquad\qquad     
    s = \begin{cases} 2m & i_3 \in I_\bl \\ 2m-1 & i_3 \in I_\wh.    
\end{cases} \label{defrs}\ee    
It is also clear that (\ref{feqn}) implies in particular that    
\be r < s \label{rnq},\ee 
for if not, the images of the three vectors $w^{n_1}\lambda_1,w^{n_2}\lambda_2,w^{n_3}\lambda_3$ would lie strictly inside some half-plane and certainly could not sum to zero.    
    
The remainder of the proof, which occupies the rest of this section, is structured as follows: Lemma \ref{qmonlemma} will re-express \ref{thi}) as a statement about the occurrence of quadratic monomials in $\chi_q(V_{i,a})$. Then \ref{thi}) $\Rightarrow$ \ref{thii}) will be an immediate corollary of proposition \ref{1impliesrule}, while \ref{thii}) $\Rightarrow$ \ref{thi}) is the content of propositions \ref{routeprop} and \ref{inchiq}.     
    
\end{proofof}    
    
\begin{lemma}\label{qmonlemma}    
The $q$-character     
\be \chi_q\left( V_{i,a} \otimes V_{j,b} \otimes V_{k,c} \right)    
=\chi_q( V_{j,b}) \chi_q( V_{i,a}) \chi_q( V_{k,c} )\ee    
can include the monomial $1$ only if $b=aq^{r-h}$ and $c=aq^{s}$ for some $r,s \in \mathbb Z$. Suppose, without loss of generality, that $s\geq 0\geq r-h$. (If not, rearrange the factors.) Then the monomial $1$ is present if and only if $\cq$ contains the quadratic monomial    
\be  Y_{\bar\jmath,bq^{h}} Y_{k,c}^{-1} .\ee  
\end{lemma}    
\begin{proof}
Assume that there exist monomials 
\be m_j \quad \mbox{in} \quad \chi_q( V_{j,b}), \qquad m_i \quad \mbox{in} \quad \chi_q( V_{i,a}),\qquad m_k \quad \mbox{in} \quad \chi_q( V_{k,c} )\ee
such that 
\be 1 = m_j m_i m_k \,. \label{mimjmk1}\ee
It follows from corollary \ref{corollary1fund} that $m_i$, $m_j$, and $m_k$ differ from $1$. Thus, eq. (\ref{mimjmk1}) can only hold by virtue of a complete cross-cancellation of all the factors of the three monomials. Since by lemma \ref{lemmacompact}, $m_i$, $m_j$ and $m_k$ each have compact support of length $h$ and respective bases $a$, $b$ and $c$, such cross-cancellation can occur only if $b=aq^{r-h}$ and $c=aq^{s}$ for some $r,s \in \mathbb Z$, thus proving the first part of the lemma. 



As for the second part, we first prove that one of the three monomials has to be the highest weight monomial of the $q$-character where it appears while another one has to be the lowest weight monomial of the $q$-character where it appears. Suppose for a contradiction that all three monomials were right negative. Since the product of two right negative monomials is obviously right negative, it would follow that $m_j m_i m_k$ is right negative and therefore not equal to $1$, a contradiction. Suppose similarly that they were all left positive: then $m_j m_i m_k$ would be left positive, a contradiction. By lemma \ref{lemmarightneg}, the only monomial in the $q$-character of a fundamental representation that is not right negative (resp. left positive) is its highest weight monomial (resp. its lowest weight monomial). 

Now it follows that the only solution to (\ref{mimjmk1}) that is also compatible with the assumption that 
\be r-h \leq 0 \leq s \label{rapineq}\ee
is $m_j = Y_{\bar{\jmath}, aq^r}^{-1}$, $m_i= Y_{\bar{\jmath}, aq^r} Y_{k,aq^s}^{-1}$ and $m_k = Y_{k,aq^s}$. 
Indeed, we know that one of the three monomials, $m_i$, $m_j$ or $m_k$, has to be the quadratic monomial obtained by multiplying the inverses of the other two, namely the one which is the highest weight monomial of its $q$-character and the one which is the lowest. By lemma \ref{lemmarightneg}, this quadratic monomial should be both right negative and left positive. Assuming that (\ref{rapineq}) holds 
thus implies $m_j \neq Y_{i,a}^{-1} Y_{\bar{k},aq^{s+h}}$ and $m_k \neq Y_{j,aq^{r-h}}^{-1} Y_{\bar{\imath},aq^h}$. Furthermore, by lemma \ref{lemmacompact}, assuming that (\ref{rapineq}) holds also implies that $m_j \neq Y_{\bar{\imath},aq^h} Y_{k, aq^s}^{-1}$ and $m_k \neq Y_{\bar{\jmath},aq^r} Y_{i,a}^{-1}$ since $m_j$ and $m_k$, as monomials in $\chi_q(V_{j,aq^{r-h}})$ and $\chi_q(V_{k,aq^s})$ respectively, should have compact supports of length $h$ and respective bases $aq^{r-h}$ and $aq^s$. Therefore, it is clear that the quadratic monomial is $m_i$. Finally, lemma \ref{lemmacompact} implies that $m_i$, as a monomial of $\chi_q(V_{i,a})$, has compact support of length $h$ and base $a$ and hence that $m_i \neq Y_{j,aq^{r-h}}^{-1} Y_{\bar{k},aq^{s+h}}$. \finproof
\end{proof}

\begin{figure}    
\begin{center}    
\begin{tikzpicture}    
\matrix (m) [matrix of math nodes, row sep=3em,    
column sep=-1em, text height=1.5ex, text depth=0.25ex]    
{           & \YY 1 0                        & \\    
            & \YY 2 1  \MM 1 2               & \\    
            & \YY 3 2 \YY 4 2  \MM 2 3      & \\    
  \YY 3 2  \MM 4 4  &     & \YY 4 2  \MM 3 4   \\    
            & \YY 2 3  \MM 3 4  \MM 4 4      & \\    
            & \YY 1 4  \MM 2 5               & \\    
          & \MM 1 6                        & \\};    
\path[->,font=\scriptsize]    
(m-1-2) edge node[fill=white,inner sep=2pt] {$ \goin 1 1 $} (m-2-2)     
(m-2-2) edge node[fill=white,inner sep=2pt] {$ \goin 2 2 $} (m-3-2)    
(m-3-2) edge node[fill=white,inner sep=2pt] {$ \goin 4 3  $} (m-4-1)    
(m-3-2) edge node[fill=white,inner sep=2pt] {$ \goin 3 3  $} (m-4-3)    
(m-4-1) edge node[fill=white,inner sep=2pt] {$ \goin 3 3  $} (m-5-2)    
(m-4-3) edge node[fill=white,inner sep=2pt] {$ \goin 4 3  $} (m-5-2)    
(m-5-2) edge node[fill=white,inner sep=2pt] {$\goin 2 4 $}  (m-6-2)    
(m-6-2) edge node[fill=white,inner sep=2pt] {$\goin 1 5 $}  (m-7-2);    
\end{tikzpicture}$\quad$    
\begin{tikzpicture}    
\matrix (m) [matrix of math nodes, row sep=3em,    
column sep=-2em, text height=1.5ex, text depth=0.25ex]    
{           & \lambda_1                       & \\    
            & w\lambda_2  - w \lambda_1               & \\    
            & w\lambda_3 + w\lambda_4  -w^2 \lambda_2      & \\    
  w\lambda_3 -w^2 \lambda_4  &     & w\lambda_4 -w^2 \lambda_3   \\    
            & w^2 \lambda_2  - w^2 \lambda_3  -w^2 \lambda_4      & \\    
            & w^2 \lambda_1  -w^3 \lambda_2                & \\    
            & -w^3 \lambda_1                        & \\};    
\path[<->,font=\scriptsize]    
(m-1-2) edge node [auto] {$=$} (m-2-2)     
(m-2-2) edge node [auto] {$=$}(m-3-2)    
(m-3-2) edge node [auto] {$=$}(m-4-1)    
(m-3-2) edge node [below left] {$=$} (m-4-3)    
(m-4-1) edge node [auto] {$=$} (m-5-2)    
(m-4-3) edge node [above left] {$=$} (m-5-2)    
(m-5-2) edge node [auto] {$=$}  (m-6-2)    
(m-6-2) edge node [auto] {$=$} (m-7-2);    
\end{tikzpicture}    
\end{center}    
\be\btp \draw[draw=white,double=black,very thick] (2,0) -- ++(1,0) -- ++(60:1) ++(60:-1) -- ++(-60:1) ;    
\filldraw[fill=white] (3,0) circle (1mm) node [right] {$2$};    
\filldraw[fill=black] (2,0) circle (1mm)  node [below] {$1$};    
\filldraw[fill=black] (3,0)++(60:1) circle (1mm)  node [right] {$3$};;    
\filldraw[fill=black] (3,0)++(-60:1) circle (1mm) node [right] {$4$};;    
\etp\nn\ee    
\caption{Proposition \ref{1impliesrule} illustrated for the representation $V_{1,a}$ of $U_q(\widehat{\mathfrak d}_4)$. On the left is the graph of the character $\chi_q(V_{1,a})$; the edge label $\goin i n$ denotes multiplication by $A_{i,aq^n}^{-1}$. On the right are the corresponding identities involving the Coxeter orbits of fundamental weights.\label{charfig}}      
\end{figure}
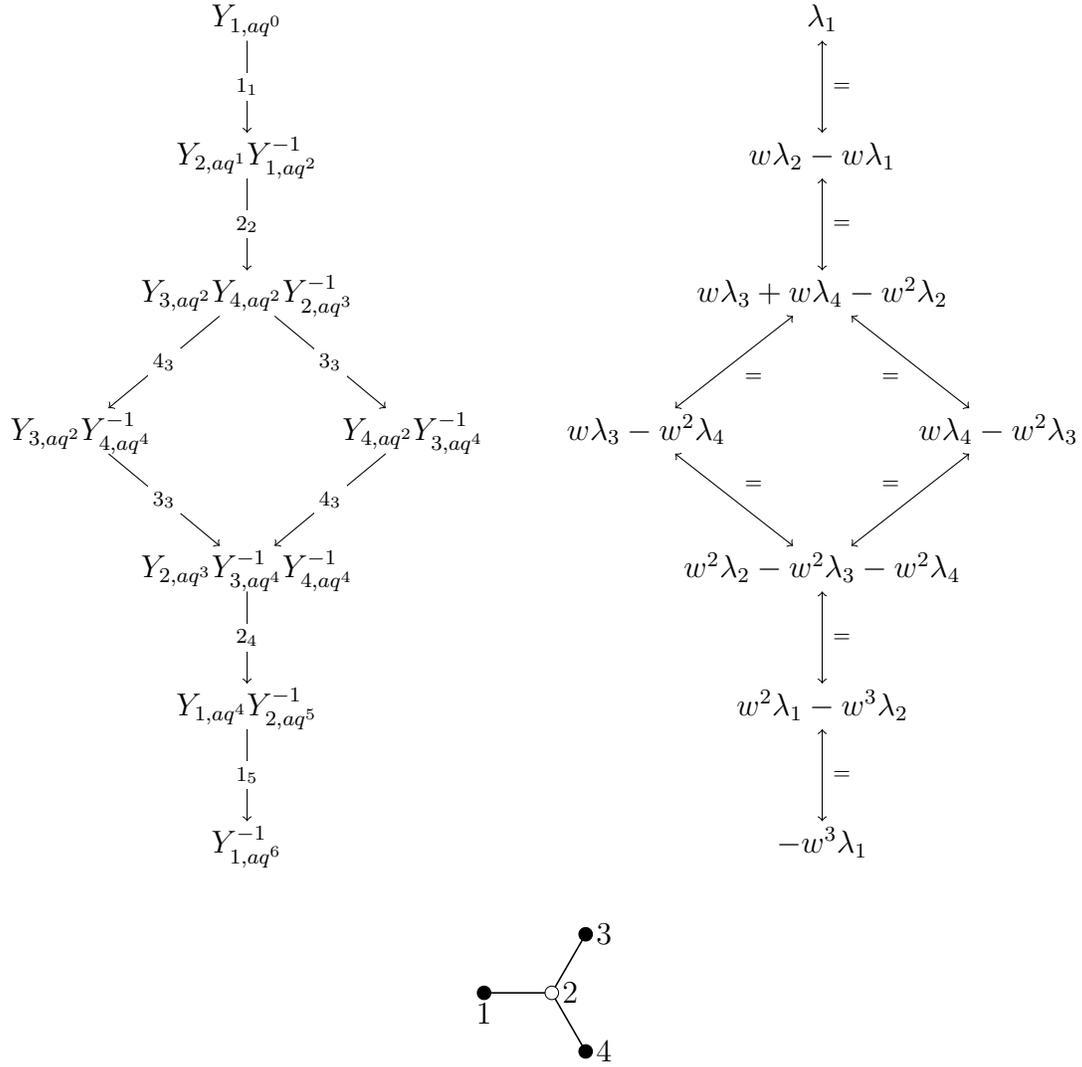    
    
\begin{proposition} \label{1impliesrule} For any given $i\in I$, choose a two-colouring of the    
Dynkin diagram of $\g$ such that $i\in I_\bl$. Then $q$-character $\chi_q\left(V_{i,a}\right)$ contains the monomial    
\be Y_{j_1,aq^{r_1}} \dots Y_{j_u,aq^{r_u}} Y^{-1}_{k_1,aq^{s_1}}    
\dots Y^{-1}_{k_v,aq^{s_v}} \ee    
only if    
\be \lambda_i = w^{n_1} \lambda_{j_1} + \dots + w^{n_u} \lambda_{j_u}     
              - w^{m_1} \lambda_{k_1} - \dots - w^{m_v} \lambda_{k_v}    
\label{lid}\ee    
where     
\be r_x = \begin{cases} 2n_x & j_x\in I_\bl \\ 2n_x-1 & j_x \in I_\wh    
\end{cases}\qquad\qquad     
    s_x = \begin{cases} 2m_x & k_x\in I_\bl \\ 2m_x-1 & k_x \in I_\wh    
\end{cases}.    
\label{rnrel}\ee    
\end{proposition}    
\begin{proof}    
We must show that each monomial in $\cq$ is associated with an identity of the form (\ref{lid}) in the fashion specified. This is certainly true of the highest monomial, which is associated with the trivial identity:    
\be Y_{i,a}  \quad \longleftrightarrow \quad \lambda_i = \lambda_i.\ee    
We know that all monomials in $\cq$ are of the form (\ref{onlyainvs}). So suppose that, for some integer $k\geq 0$, we have successfully demonstrated the required identity for all monomials in $\cq$ that are $k$ lowering steps, in the sense of (\ref{lowering}), from $Y_{i,a}$. Let $m'\in \cq$ be any monomial $k+1$ steps away from $Y_{i,a}$.  By theorem \ref{lowprop}, we have that $m' = m A_{j,aq^{r+1}}^{-1}$, for some monomial $m\in\cq$ that is $k$ lowering steps from $Y_{i,a}$ and that has as a factor $Y_{j,aq^r}$. By supposition,  the identity to which $m$ is associated thus contains a summand $+w^n \lambda_j$, where $n$ and $r$ are related as in (\ref{rnrel}). We associate the lowering operation (\ref{lowering}) in the direction of the simple root $\alpha_j$ with one of the following re-writings of $\lambda_j$, to be chosen according to the colour of the node $j\in    
I$:     
\be \lambda^\bl_j \rr w \sum_{\nbr k j} \lambda_k^\wh    
- w \lambda_j^\bl \ee    
\be \lambda^\wh_j \rr   \sum_{\nbr k j} \lambda_k^\bl    
- w \lambda_j^\wh . \ee    
That these are identities follows from (\ref{loweringrel}). It is straightforward to check that they produce precisely the terms required for the resulting identity to be that associated to $m'=m A_{j,aq^{r+1}}^{-1}$ as the proposition requires.  This completes the inductive step, and the result follows by induction on $k$.    
\finproof    
\end{proof}    
    
An example is shown in figure \ref{charfig}. Now, in particular, the quadratic monomials required in lemma \ref{qmonlemma} correspond to identities of the form     
\be \lambda_i = w^n \lambda_{\bar\jmath} - w^m \lambda_k\label{fuzsol}.\ee    
This completes the proof of the \ref{thi}) $\Rightarrow$ \ref{thii}) part of theorem \ref{sec3thm}. It remains to prove the converse. 
In view of the preceding proposition, it is clear that what underpins this whole approach is the similarity between the definition    
\be A_{i,a} = Y_{i,aq^{-1}} Y_{i,aq} \prod_{\nbr j i } Y_{j,a}^{-1}\ee    
and the identities    
\be 0 = \lambda^\bl_i + w\lambda_i^\bl - w \sum_{\nbr j i } \lambda_j^\wh ,\qquad    
 0 = \lambda^\wh_i + w\lambda_i^\wh -  \sum_{\nbr j i } \lambda_j^\bl.\ee    
In trying to pass from a solution to the fusing rule to a monomial in the $q$-character, the first problem is thus to express the solution explicitly in terms of these identities.

\begin{figure}
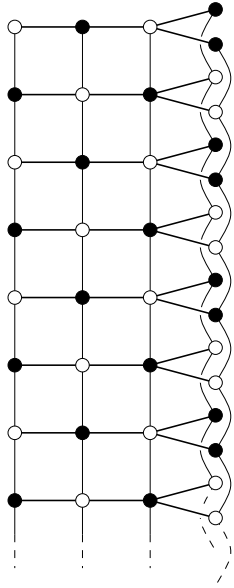
     
\be\nn \btp    
\foreach \y in {1,2,3} { \draw  (\y,0) -- ++(0,-7.5);    
                         \draw[dashed] (\y,-7.5) -- ++(0,-.5);}      
    
    
\foreach \x in {0,1,2,3,4,5,6}     
{    
\draw (3,-\x)++(-15:1) ..  controls ++(.3,-.5) .. ++(0,-1);    
\draw (3,-\x)++(+15:1) .. controls ++(-.3,-.5) .. ++(0,-1);    
}    
\draw[dashed] (3,-7)++(-15:1) ..  controls ++(.3,-.5) .. ++(0,-1);    
\draw[dashed] (3,-7)++(+15:1) .. controls ++(-.3,-.5) .. ++(0,-1);

\foreach \x in {0,2,4,6}     
{    
\draw[draw=white,double=black,very thick] (1,-\x) -- ++(2,0) -- ++(15:1) ++(15:-1) -- ++(-15:1) ;    
\filldraw[fill=white] (1,-\x) circle (1mm);    
\filldraw[fill=white] (3,-\x) circle (1mm);    
\filldraw[fill=black] (2,-\x) circle (1mm) ;    
\filldraw[fill=black] (3,-\x)++(15:1) circle (1mm);    
\filldraw[fill=black] (3,-\x)++(-15:1) circle (1mm);    
}    
\foreach \x in {1,3,5,7}    
{    
\draw[draw=white,double=black,very thick]   (1,-\x) -- ++(2,0) -- ++(15:1) ++(15:-1) -- ++(-15:1) ;    
\filldraw[fill=black] (1,-\x) circle (1mm);    
\filldraw[fill=black] (3,-\x) circle (1mm);    
\filldraw[fill=white] (2,-\x) circle (1mm) ;    
\filldraw[fill=white] (3,-\x)++(15:1) circle (1mm);    
\filldraw[fill=white] (3,-\x)++(-15:1) circle (1mm);    
}     
\etp \ee    
\caption{\label{ihatpic1} The bipartite graph $\Ih$ in the case $\g = \mathfrak d_5$.}    
\end{figure}

We begin by introducing some useful scaffolding. By a slight abuse of notation, let $I$ be the Dynkin diagram of $\g$, and consider the product graph $\Ih = I\times \{0,1,2,\dots \}$. The two-colouring of $I$ extends to a two-colouring $\Ih=\Ihw \sqcup \Ihb$ of the infinite graph. We picture $\Ih$ as a vertical stack of copies of $I$, and will refer to each copy of $I$ as a \emph{row} and to the set of nodes $(j,0), (j,1),\dots$ for any fixed $j\in I$ as a \emph{column}.     
(Figure \ref{ihatpic1} illustrates an example.)    
    
The black nodes of $\Ih$ are those of the form $(i,2n)$, $i\in I_\bl$ and $(i,2n-1)$, $i\in I_\wh$.     
We associate to each black node a factor $Y$ in the obvious way:    
\be (i,r) \in \Ihb \mapsto Y_{i,aq^r}.\ee    
We also associate to each black node $(i,r)$ a term $\y{i,r}$ of the form $w^n\lambda_i$, defined as follows:    
\be \y{i\in I_\bl,2n} :=  w^n\lambda_i^\bl, \qquad \y{i\in I_\wh,2n-1} :=  w^n\lambda_i^\wh.    
\ee    
The white nodes of $\Ih$ are those of the form $(i,2n)$, $i\in I_\wh$ and $(i,2n-1)$, $i\in I_\bl$. We associate to each white node $(i,r)$ the factor $A_{i,aq^r}$, and also an identity $\aaa{}$ among the terms $w^n\lambda_i$ at the neighbouring black nodes:    
\bea 0 = \aaa{i\in I_\wh,2n} &:=& w^n \left(\lambda^\wh_i + w\lambda_i^\wh -  \sum_{\nbr j i } \lambda_j^\bl\right)\\    
0 = \aaa{i\in I_\bl,2n-1} &:=& w^{n-1} \left(\lambda^\bl_i + w\lambda_i^\bl -  w\sum_{\nbr j i } \lambda_j^\wh\right);\eea    
that is, simply,    
\be  \aaa{i,r} := \y{i,r-1} + \y{i,r+1} - \sum_{\nbr j i} \y{j,r}.\ee    
    
Let $c$ and $g$ be integer-valued functions defined on the black and white nodes respectively     
\be c : \Ihb \rightarrow \mathbb Z; \quad (i,n) \mapsto c_i^n,\ee    
\be g: \Ihw  \rightarrow \mathbb Z; \quad (i,n) \mapsto 
 g_i^n.\ee    
One may then ask: when does the coefficient of a term $\y{i,n}$, with $n>0$, vanish in the expression     
\be \label{Eca}\mathscr E(c,g) := \sum_{(j,r) \in \Ihb} c_j^r \y{j,r} 
 - \sum_{(j,r) \in \Ihw} g_j^r \aaa{j,r}\, ? \ee    
Or, equivalently, when is the factor $Y_{i,aq^n}$ absent from the monomial     
\be m(c,g) := \left(\prod_{(j,r) \in \Ihb} \left(Y_{j,aq^r}\right)^{c_j^r}  \right)    
 \left(\prod_{(j,r) \in \Ihw} \left(A_{j,aq^r}\right)^{ - g_j^r}\right)\,\text{?}\ee     
It is clear that the answer is: if and only if    
\be g_i^{n-1} + g_i^{n+1} - \sum_{\nbr j i} g_j^n = 
 c_i^n.\label{sseqn}\ee      
Let us regard $c$ as a fixed source term. Then it is possible to satisfy (\ref{sseqn}) at every black node $(i,n)$ with $n>0$ by choosing an appropriate $g$. Assume that sufficiently far down the graph the source vanishes, i.e. that there is an $N$ such that $c_i^n=0$ for all $n>N$. Then, furthermore, the solution is unique if we specify also that $g_i^n = 0$ for all $n>N$, because the equation (\ref{sseqn}) at each row $n$ fixes uniquely the $g_i^{n-1}$ in the row above.     
    
\begin{proposition}\label{routeprop}    
Choose the two-colouring of $I$ such that $i_1\in I_\bl$. Suppose that we have a solution to the fusing equation (\ref{feqn}), written, as in (\ref{fusrew}), in the form    
\be \lambda_{i_1} - w^n \lambda_{\bar\imath_2} + w^m \lambda_{i_3} = 0,\label{fz}\ee    
with $n, m \in \mathbb Z$ subject to (\ref{ineqnm}). Then there exists a unique $g: \Ihw  \rightarrow \mathbb Z$
such that
\be\label{yas}   Y_{\bar\imath_2,aq^r} Y_{i_3,aq^s}^{-1} = Y_{i_1,a} \prod_{(j,t) \in \Ihw} \left(A_{j,aq^t}\right)^{-g_j^t} ,\ee    
where $r, s \in \mathbb Z$ are as in (\ref{defrs}), and such that, for some $N \in \mathbb N$, $g^n_i =0$ for all $n>N$.
\end{proposition}     
\begin{proof} Let $c$ be the source function that vanishes everywhere except    
\be c_{i_1}^0 = +1, \qquad c_{\bar\imath_2}^r = -1, \qquad c_{i_3}^s = +1. \label{cdef}\ee    
Note that then (\ref{fz}) is $\mathscr E(c,0) =0$. Consider solving (\ref{sseqn}) for $g$ in the manner given above. The resulting expression $\mathscr E(c,g)$ has by construction no terms $\y{i,n}$ with $n>0$. So it can only be a linear combination of the $\y{i,0}= \lambda_i^\bl$ and~\footnote{For this proof only, we consider working on $I \times \{ -1, 0, 1 \dots \}$} $\y{i,-1}= \lambda_i^\wh$.
But of course $\mathscr E(c,g)=0$ identically, since all we have done is to add various re-writings of zero (the $\aaa{}$'s) to an expression (\ref{fz}) which was zero to begin with. Therefore, since the $\lambda_i$ are linearly independent, the identity $\mathscr E(c,g)=0$ must be trivial, in the sense that the expression on the right-hand side of (\ref{Eca}) consists entirely of cancelling pairs of terms and vanishes without appealing to properties of the Coxeter element. Consequently, we have also that $m(c,g)=1$, which, on rearranging, is (\ref{yas}) as required.\finproof    
\end{proof}     
\begin{figure}
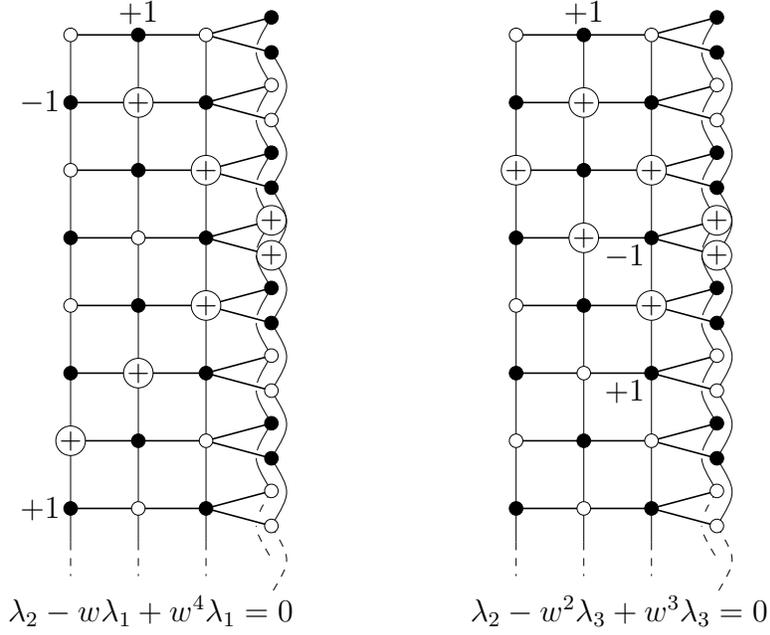
    
\be \begin{matrix}\btp    
\foreach \y in {1,2,3} { \draw  (\y,0) -- ++(0,-7.5);    
                         \draw[dashed] (\y,-7.5) -- ++(0,-.5);}      
\foreach \x in {0,1,2,3,4,5,6}     
{    
\draw (3,-\x)++(-15:1) ..  controls ++(.3,-.5) .. ++(0,-1);    
\draw (3,-\x)++(+15:1) .. controls ++(-.3,-.5) .. ++(0,-1);    
}    
\draw[dashed] (3,-7)++(-15:1) ..  controls ++(.3,-.5) .. ++(0,-1);    
\draw[dashed] (3,-7)++(+15:1) .. controls ++(-.3,-.5) .. ++(0,-1);    
    
\foreach \x in {0,2,4,6}     
{    
\draw[draw=white,double=black,very thick] (1,-\x) -- ++(2,0) -- ++(15:1) ++(15:-1) -- ++(-15:1) ;    
\filldraw[fill=white] (1,-\x) circle (1mm);    
\filldraw[fill=white] (3,-\x) circle (1mm);    
\filldraw[fill=black] (2,-\x) circle (1mm) ;    
\filldraw[fill=black] (3,-\x)++(15:1) circle (1mm);    
\filldraw[fill=black] (3,-\x)++(-15:1) circle (1mm);    
}    
\foreach \x in {1,3,5,7}    
{    
\draw[draw=white,double=black,very thick]   (1,-\x) -- ++(2,0) -- ++(15:1) ++(15:-1) -- ++(-15:1) ;    
\filldraw[fill=black] (1,-\x) circle (1mm);    
\filldraw[fill=black] (3,-\x) circle (1mm);    
\filldraw[fill=white] (2,-\x) circle (1mm) ;    
\filldraw[fill=white] (3,-\x)++(15:1) circle (1mm);    
\filldraw[fill=white] (3,-\x)++(-15:1) circle (1mm);    
}     
\draw (2,0) node [above] {$+1$};     
\draw (1,-1) node [left] {$-1$};     
\draw (1,-7) node [left] {$+1$};    
\draw (1,-6) node [inner sep = 0mm,circle,draw,fill=white] {\small +};     
\draw (2,-5) node [inner sep = 0mm,circle,draw,fill=white] {\small +};     
\draw (3,-4) node [inner sep = 0mm,circle,draw,fill=white] {\small +};     
\draw (3,-3)++(15:1) node [inner sep = 0mm,circle,draw,fill=white] {\small +};     
\draw (3,-3)++(-15:1) node [inner sep = 0mm,circle,draw,fill=white] {\small +};     
\draw (3,-2) node [inner sep = 0mm,circle,draw,fill=white] {\small +};     
\draw (2,-1) node [inner sep = 0mm,circle,draw,fill=white] {\small +};     
\etp &\qquad\qquad&    
 \btp     
\foreach \y in {1,2,3} { \draw  (\y,0) -- ++(0,-7.5);    
                         \draw[dashed] (\y,-7.5) -- ++(0,-.5);}      
\foreach \x in {0,1,2,3,4,5,6}     
{    
\draw (3,-\x)++(-15:1) ..  controls ++(.3,-.5) .. ++(0,-1);    
\draw (3,-\x)++(+15:1) .. controls ++(-.3,-.5) .. ++(0,-1);    
}    
\draw[dashed] (3,-7)++(-15:1) ..  controls ++(.3,-.5) .. ++(0,-1);    
\draw[dashed] (3,-7)++(+15:1) .. controls ++(-.3,-.5) .. ++(0,-1);    
    
\foreach \x in {0,2,4,6}     
{    
\draw[draw=white,double=black,very thick] (1,-\x) -- ++(2,0) -- ++(15:1) ++(15:-1) -- ++(-15:1) ;    
\filldraw[fill=white] (1,-\x) circle (1mm);    
\filldraw[fill=white] (3,-\x) circle (1mm);    
\filldraw[fill=black] (2,-\x) circle (1mm) ;    
\filldraw[fill=black] (3,-\x)++(15:1) circle (1mm);    
\filldraw[fill=black] (3,-\x)++(-15:1) circle (1mm);    
}    
\foreach \x in {1,3,5,7}    
{    
\draw[draw=white,double=black,very thick]   (1,-\x) -- ++(2,0) -- ++(15:1) ++(15:-1) -- ++(-15:1) ;    
\filldraw[fill=black] (1,-\x) circle (1mm);    
\filldraw[fill=black] (3,-\x) circle (1mm);    
\filldraw[fill=white] (2,-\x) circle (1mm) ;    
\filldraw[fill=white] (3,-\x)++(15:1) circle (1mm);    
\filldraw[fill=white] (3,-\x)++(-15:1) circle (1mm);    
}     
\draw (2,0) node [above] {$+1$};     
\draw (3,-3) node [below left=-.5mm] {$-1$};     
\draw (3,-5) node [below left=-.5mm] {$+1$};    
\draw (3,-4) node [inner sep = 0mm,circle,draw,fill=white] {\small +};     
\draw (3,-3)++(15:1) node [inner sep = 0mm,circle,draw,fill=white] {\small +};     
\draw (3,-3)++(-15:1) node [inner sep = 0mm,circle,draw,fill=white] {\small +};     
\draw (2,-3) node [inner sep = 0mm,circle,draw,fill=white] {\small +};     
\draw (3,-2) node [inner sep = 0mm,circle,draw,fill=white] {\small +};     
\draw (1,-2) node [inner sep = 0mm,circle,draw,fill=white] {\small +};     
\draw (2,-1) node [inner sep = 0mm,circle,draw,fill=white] {\small +};     
\etp \\    
\lambda_2 - w\lambda_1 + w^4 \lambda_1 = 0 && \lambda_2 - w^2 \lambda_3 + w^3 \lambda_3 = 0 \end{matrix} \nn    
\ee    
\caption{Two copies of $\Ih$ in the case $\g = \mathfrak d_5$, showing the solutions to the problem (\ref{sseqn}) for the source functions $c$ associated, as in proposition \ref{routeprop}, to the identities shown. $\oplus$ denotes a node at which $g=+1$; elsewhere $g=0$.\label{ihatpic}}    
\end{figure}    
    
The right-hand side of (\ref{yas}) is of the right form to be a monomial in $\cqo$, c.f. (\ref{onlyainvs}), but we are by no means done. A priori, it is perhaps not even clear from the procedure above that the $g_i^n$ need all be non-negative: indeed, although we stated the proposition for identities involving three terms, the obvious generalization to arbitrary identities of the form (\ref{lid}) is valid, but the resulting $g_i^n$ are not all non-negative in general. Nonetheless,
\begin{proposition}    
\label{quadrmonom}
Under the assumptions of the preceding proposition, the monomial \be Y_{\bar\imath_2,aq^r} Y_{i_3,aq^s}^{-1}\ee of (\ref{yas}) occurs in $\cqo$. \label{inchiq}    
\end{proposition}    
\begin{proof}    
First consider the following iterative procedure which generates a finite sequence $m'_0,m'_1,\dots, m'_{h-1}$ of monomials in $\cqo$. We set $m'_0= Y_{i_1,a}$. Roughly speaking, the idea is to lower fully in all black directions to obtain $m'_1$, then lower fully in all white directions to obtain $m'_2$, and so on. More precisely, suppose that for some even $p\geq 0$ we have found an $m'_p$ in $\cqo$ of the form    
\be m'_p = \left(\prod_{i\in I_\bl} Y_{i,aq^{p}}^{b_i} \right)    
          \left(\prod_{i\in I_\wh} Y_{i,aq^{p+1}}^{b_i} \right)^{-1}\label{mp} \ee    
for some non-negative integers $b_i$, $i\in I$. Certainly (c.f. \ref{onlyainvs})     
\be m'_p=Y_{i_1,a} \prod_{(j,t) \in \Ihw} \left(A_{j,aq^t}\right)^{-g'{}_j^t}\label{mp2}\ee     
for some $g'{}_j^t\geq 0$ with, in view of (\ref{mp}), $g'{}_j^t=0\, \forall\, t>p
$. Thus for all $k>0$ and $i\in I$, $m'_p A_{i,aq^{p+k}}$ is not of the form (\ref{onlyainvs}) and so cannot be in $\cqo$. Proposition \ref{yang} thus guarantees that $m'_p A_{i,aq^{p+1}}^{-b_i}$ is in $\cqo$, for $i\in I_\bl$. By similar reasoning for each black direction in turn, we have that $\cqo$ contains     
\be m'_{p+1} = m'_p \prod_{i\in I_\bl} A_{i,aq^{p+1}}^{-b_i}.\ee    
It too is of the form (\ref{mp}), but with $p$ odd and the roles of black and white exchanged. With the obvious colour swaps, we then iterate.      
    
As stated, the iteration proceeds until we arrive at the lowest monomial $m'_{h-1} = Y_{\bar\imath_1, aq^h}^{-1}$ of $\cqo$.\footnote{This sequence of ``lowering steps'' is of the general type mentioned in \cite{HernandezMinAff}, remark 2.16. Note that this particular sequence picks out a route through the graph of $\cqo$, from the highest to the lowest monomial, that avoids non-trivial $\uqslt$ Kirillov-Reshetikhin modules, in the sense that at each lowering step the relevant $\uqslt$-character is that of an (irreducible) tensor product of fundamental representations at coincident rapidity. It is also interesting to note that the monomials $m'_0, m'_1,\dots, m'_{h-1}$ have the property that the sequence of their classical weights is a permutation of the Coxeter orbit of the highest weight $\lambda_{i_1}$.}  The key observation is that, for all $p\leq h-1$, the  $g'{}_j^t$ of (\ref{mp2}) solve the problem (\ref{sseqn}) in rows $1,2\dots, p-1$, for the source function $c'$ defined to be zero everywhere except for $c'{}_{i_1}^0 = +1$, and the initial conditions $g'{}_i^0=0$ $\forall i \in I$. 
    
Note that for all $p\leq h-1$ the  $g'{}_j^t$ of (\ref{mp2}) are non-negative in rows $1,2,\dots,p-1$; this is clear from their character-theoretic construction, and is a fact about the solution to (\ref{sseqn}) for the source $c'$ and initial conditions $g'{}_i^0=0$ $\forall i \in I$ that is not otherwise manifest.

Now let $g$ and $c$ be the functions of the proof of proposition \ref{routeprop}.     
Since    
\be \forall n< r,\quad c_i^n = c'{}_i^n\ee    
and because each $g_i^n$ is determined by the values of $c$ and $g$ in rows above (when we think of solving from row 0 downwards), we have    
\be \forall n\leq r,\quad g_i^n = g'{}_i^n.\label{aeq}\ee    
In particular, the $g_i^n$ are non-negative for all $n\leq r$. On the other hand, by imagining turning the diagram upside-down and applying the same argument starting from the $+1$ source in row $s$, we conclude also that the $g_i^n$ are non-negative for all $n\geq r$. Therefore all the $g_i^n$ are non-negative. (Note that this trick would not work if $c$ were non-zero at more than three nodes.) Furthermore, again thinking of solving from row 0 downwards,    
\be c'{}_{\bar\imath_2}^r = c_{\bar\imath_2}^r+1 \quad\implies\quad    
 g'{}_{\bar\imath_2}^{r+1} = g_{\bar\imath_2}^{r+1}  +1 > 0.\label{ago}\ee    
This relation is crucial, because if we are to obtain the desired quadratic monomial (\ref{yas}), we must modify the procedure on reaching row $r$: we set $m_1=m'_1,\dots,m_{r}=m'_{r}$, but then rather than lowering $m_r$ completely in the direction $\bar\imath_2$, we want to preserve one factor of $Y_{\bar\imath_2,aq^r}$ -- and the above inequality guarantees that there \emph{is} at least one such factor. That is, if     
\be m_r = \left(\prod_{i\in I_\bl} Y_{i,aq^{r}}^{b_i} \right)    
          \left(\prod_{i\in I_\wh} Y_{i,aq^{r+1}}^{b_i} \right)^{-1}\label{mr},\ee    
supposing in what follows, for the sake of definiteness, that $\bar\imath_2\in I_\bl$, then we are guaranteed that $b_{\bar\imath_2}\geq 1$. Setting $n=b_{\bar\imath_2}$, $p=b_{\bar\imath_2}-1$ in proposition \ref{yang} we deduce that      
\be m_{r+1} := 
m_r  A^{-b_{\bar\imath_2}+1}_{\bar\imath_2, aq^{r+1}} \prod_{i\in I_\bl\setminus \{\bar\imath_2\}} A_{i,aq^{r+1}}^{-b_i} 
 =  Y_{i_1,a} \prod_{(j,v) \in \Ihw: t\leq {r+1}} \left(A_{j,aq^t}\right)^{-g{}_j^t} \label{rowr}\ee    
is a monomial in $\cqo$.
We would then like to continue to apply the above alternating black/white lowering procedure in subsequent rows, preserving the prefactor $Y_{\bar\imath_2,aq^r}$ at each step. Once more we shall argue that this is possible by a finite recursion. 
Consider a white lowering step: suppose that for some odd $p$ with $s>p\geq r+1$ we have shown that 
\be  m_p := Y_{i_1,a} \prod_{(j,t) \in \Ihw: t\leq p} \left(A_{j,aq^t}\right)^{-g{}_j^t} =
     Y_{\bar\imath_2,aq^r}\left(\prod_{i\in I_\wh}Y_{i,aq^p}^{b'_i} \right)\left(\prod_{i\in I_\bl} Y_{i,aq^{p+1}}^{b'_i}\right)^{-1}\ee    
is a monomial in $\cqo$, for certain  $b'_i\in \mathbb Z$, $i\in I$. To begin the recursion, this is certainly true for $p=r+1$, as in (\ref{rowr}). Now observe that in fact, for all $i\in I_\wh$, $b'_i= g_i^{p+1}$ (this is clear when thinking of solving for $g$ row-by-row from row 0 downwards) and that these are non-negative as noted above. Thus we can lower in all white directions as before and find that 
\be m_{p+1} := m_p\prod_{(j,p+1) \in \Ihw} \left(A_{j,aq^{p+1}}\right)^{-g{}_j^{p+1}}\ee
is also a monomial in $\cqo$. This completes the white inductive step. For the black step, lowering in the directions $I_\bl\setminus\{\bar\imath_2\}$ works in exactly the same way. It remains only to check that the lowering step in the direction $\bar\imath_2$ is also valid: but this is clear because $m_{p+1}$ is an $\bar\imath_2$-dominant monomial and $\beta_{\bar\imath_2}(m_{p+1}) = Y_{\bar\imath_2,aq^r} Y_{\bar\imath_2,aq^{p+1}}^{n}$ with $n=g_{\bar\imath_2}^{p+2}\geq 0$, which is still of the correct form to apply proposition \ref{yang}. Iterating, we have that every monomial in the sequence
\be  Y_{i_1,a} \prod_{(j,v) \in \Ihw: t\leq p} \left(A_{j,aq^t}\right)^{-g{}_j^t}    
               \quad\text{for} \quad p = 1,2,\dots,s\ee
is in $\cqo$.   
Finally then, at row $s$, we indeed arrive at    
\be  Y_{i_1,a} \prod_{(j,t) \in \Ihw} \left(A_{j,aq^t}\right)^{-g{}_j^t} =  Y_{\bar\imath_2,aq^r}  Y_{i_3,aq^s}^{-1},\ee    
which is the required monomial.\finproof    
\end{proof}    
    
    
\section{Outlook}\label{outlook}    
It is an immediate corollary of our main result, theorem \ref{sec3thm}, that Dorey's rule provides a \emph{necessary} condition for $\Hom_{\uqgh}\left( V_{i,a} \otimes V_{j,b} \otimes V_{k,c}, \mathbb   C\right) \neq 0$. We have not, however, given a general proof here of sufficiency; and it may be that such a proof would require more knowledge about the structure of fundamental $\uqgh$-modules than their $q$-characters alone provide.  The correct statement should be the following. Under the conditions of theorem \ref{sec3thm}, 
the ordered triple of vectors  $(w^{n_1} \lambda_{i_1}, w^{n_2} \lambda_{i_2}, w^{n_3} \lambda_{i_3})$ can be said to be either \emph{cyclic} or \emph{acyclic}
according to the order in which their projections occur in the oriented $s=1$ eigenplane of $w$, c.f. (\ref{thetadef}). In the example following the theorem, $(w^{-2}\lambda_{\mathbf{\bar l}},\lambda_\mathbf{L},w^5\lambda_{\mathbf{h}})$ is cyclic, for instance. It should be that, in the cyclic case, $\Hom_{\uqgh} (\mathbb C,  V_{i_1,a_1} \otimes V_{i_2,a_2} \otimes V_{i_3,a_3}) \neq 0$ and $\Hom_{\uqgh} (V_{i_3,a_3} \otimes V_{i_2,a_2} \otimes V_{i_1,a_1}, \mathbb C) \neq 0$.     
(For the $\mathfrak{a}$- and $\mathfrak d$-series, one may verify that this statement indeed unpacks to give theorems 6.1 and 7.1 of \cite{CP95}. There the proof proceeds by induction on the rank, and relies on specific properties of these root systems.)  
  
Now, as mentioned in section 2, there is a ``current'' Hopf algebra structure on $\uqgh$, originally due to Drinfel'd. It restricts, over the quantum loop algebra, to the following relations:    
\be \Delta (\phi^\pm_i(u)) = \phi^\pm_i(u) \otimes \phi^\pm_i(u) \label{deltaphi}\ee    
\be \Delta (x^+_i(u)) =1 \otimes x^+_i (u) +  x^+_i(u) \otimes \phi^-_i(1/u) \label{deltax+}\ee    
\be \Delta (x^-_i(u)) = x^-_i(u) \otimes 1 + \phi^+_i(1/u) \otimes x^-_i(u) \label{deltax-}\ee    
\be S (\phi^\pm_i(u)) = \phi^\pm_i(u)^{-1}\ee    
\be  S (x^+_i(u)) =- x^+_i (u) \phi^-_i(1/u)  \qquad  S (x^-_i(u)) = -  \phi^+_i(1/u) x^-_i(u) \ee    
\be \epsilon (\phi^\pm_i(u)) = 1  \qquad \epsilon (x^\pm_i(u)) = 0, \ee    
where $\Delta$ is the coproduct, $S$ the antipode and $\epsilon$ the counit. This Hopf algebra structure is twist-equivalent to the standard one in a sense given in \cite{EKP}; note that the infinite sums on the right of the coproducts above require careful interpretation \cite{HernandezFusionI,HernandezFusionII,Grosse}.    
With respect to this ``current'' Hopf algebra structure, it is clear that the singlet state in  $V_{i_1,a_1} \otimes V_{i_2,a_2} \otimes V_{i_3,a_3}$ must be of the form $\ket{\smash{Y_{\bar \imath_1,a_1q^h}^{-1}}} \otimes \ket{Y_{\bar\imath_1,a_1q^h}Y^{-1}_{i_3,a_3}} \otimes \ket{Y_{i_3,a_3}}$ -- where the first and last tensor factors are the lowest and highest weight vectors of the respective representations, and the middle factor is an eigenvector of $\phi^\pm_i(u)$ with l-weight corresponding to the monomial shown.

Finally, let us remark that it would be interesting to investigate whether generalizations of our results exist for representations other than the fundamental ones (as was suggested in \cite{EKMY} based on the structure of local charges in certain integrable sigma models). The natural candidates are the Kirillov-Reshetikhin modules $W^{(k)}_{i,a}$, which can be thought of as the ``minimal affinizations'' \cite{Chari:1994rv} of the highest weight representations $V_{k\lambda_i}$ of $\g$ and for which the Frenkel-Mukhin algorithm is known to work \cite{Nakajima:2002fz,HernandezKR}. The form of our arguments suggests that such generalizations may be possible, perhaps using the braid group actions of \cite{BeckBraid,ChariBraid} to lift the periodicity of the Coxeter element.    
    
\vspace{1cm}   
  
\textsl{Acknowledgements.} \,\, We are grateful to Patrick Dorey and Niall MacKay for valuable discussions and suggestions. During much of the preparation of this work, C.A.S.Y. was funded by the Leverhulme trust and R.Z. by an EPSRC postdoctoral fellowship.  C.A.S.Y. is funded by a fellowship from the Japan Society for the Promotion of Science.

\bibliography{DR_CMPv3}    
\bibliographystyle{alpha}    
\end{document}